\documentclass{amsart}
\makeatletter
\let\@wraptoccontribs\wraptoccontribs
\makeatother

\pdfoutput=1

\title{The enumerative geometry and arithmetic of banana nano-manifolds}
\author{Jim Bryan and Stephen Pietromonaco}
\contrib[Appendix with ]{Mike Roth}

\date{\today}

\usepackage{amsmath}
\usepackage{amsmath,amsthm,amsfonts}
\usepackage{mathtools}
\usepackage{times}
\usepackage{graphicx}
\usepackage{tikz-cd}
\usepackage{verbatim}
\usepackage{url}

\usepackage[margin=2.5cm]{geometry}
\usepackage{tikz, amsmath, amssymb}
\usetikzlibrary{intersections, calc, positioning, decorations.pathmorphing, decorations.pathreplacing,decorations.shapes} 

\usepackage{amsfonts,tikz,tikz-layers}
\usetikzlibrary{fadings,quotes, shapes,calc,decorations.markings}

\usetikzlibrary{shapes,shapes.geometric,positioning, arrows, arrows.meta}
\usetikzlibrary{backgrounds}
\usetikzlibrary{3d} 
\usepackage{tikz-3dplot}

\newtheorem{theorem}{Theorem}
\newtheorem{proposition}[theorem]{Proposition}

\newtheorem{lemma}[theorem]{Lemma}
\newtheorem{corollary}[theorem]{Corollary}

\newtheorem{question}{Question}[theorem]

\theoremstyle{definition}

\newtheorem{def-theorem}[theorem]{Definition-Theorem}
\newtheorem{remark}[theorem]{Remark}

\newcommand{\CC} {{\mathbb C}}          
\newcommand{\RR} {{\mathbb R}}		
\newcommand{\ZZ} {{\mathbb Z}}		
\newcommand{\QQ} {{\mathbb Q}}		
\newcommand{\PP} {{\mathbb P}}		
\newcommand{\HH} {{\mathbb H}}		
\newcommand{\FF} {{\mathbb F}}		
\newcommand{\DD} {{\mathbb D}}		
\newcommand{\XX} {{\mathbb X}}		
\renewcommand{\AA} {{\mathbb A}}

\newcommand{\Tr}{\operatorname{tr}}

\newcommand{\Hilb}{\operatorname{Hilb}}

\newcommand{\Spec}{\operatorname{Spec}}

\newcommand{\Frob}{\operatorname{Frob}}
\newcommand{\Gr}{\operatorname{Gr}}
\newcommand{\Sp}{\operatorname{Sp}}
\newcommand{\diag}{\operatorname{diag}}
\newcommand{\Gal}{\operatorname{Gal}}

\newcommand{\sing}{\mathsf{sing}}
\newcommand{\con}{\mathsf{con}}
\newcommand{\loc}{\mathsf{loc}}
\newcommand{\Sch}{\mathsf{Sch}}
\newcommand{\ban}{\mathsf{ban}}
\newcommand{\Ver}{\mathsf{Ver}}
\newcommand{\et}{\mathsf{\acute{e}t}}

\newcommand{\DT}{\mathsf{DT}}

\newcommand{\OO}{\mathcal{O}}

\newcommand{\Bl}{\operatorname{Bl}}
\newcommand{\dvec}{\vec{\mathbf{d}}}

\newcommand{\half}{\frac{1}{2}}

\newcommand{\Fhat}{\widehat{F}}

\newcommand{\Pic}{\operatorname{Pic}}

\newcommand{\XtildeN}{\widetilde{X}_{N}}

\newcommand{\Stilde}{\widetilde{S}}
\newcommand{\Sprimetilde}{\widetilde{S}'}
\newcommand{\Deltatilde}{\widetilde{\Delta }}
\newcommand{\Ftilde}{\widetilde{F}}
\newcommand{\Gammatilde}{\widetilde{\Gamma }}
\newcommand{\atilde}{\widetilde{a}}
\newcommand{\btilde}{\widetilde{b}}
\newcommand{\ctilde}{\widetilde{c}}
\newcommand{\ftilde}{\widetilde{f}}
\newcommand{\fprimetilde}{\widetilde{f}'}
\newcommand{\epsilontilde}{\widetilde{\epsilon}}
\newcommand{\deltatilde}{\widetilde{\delta}}
\newcommand{\gammatilde}{\widetilde{\gamma }}
\newcommand{\betatildesubd}{\widetilde{\beta}_{\dvec }(k) }
\newcommand{\divides}{\mathbin{\Big{|}}}

\begin{document}
\begin{abstract}
A banana manifold is a Calabi-Yau threefold fibered by Abelian
surfaces whose singular fibers contain banana configurations: three
rational curves meeting each other in two points. A nano-manifold is a
Calabi-Yau threefold $X$ with very small Hodge numbers:
$h^{1,1}(X)+h^{2,1}(X)\leq 6$. We construct four rigid banana
nano-manifolds $\XtildeN$, $N\in \{5,6,8,9 \}$, each with Hodge
numbers given by $(h^{1,1},h^{2,1})=(4,0)$.

We compute the Donaldson-Thomas partition function for banana curve
classes and show that the associated genus $g$ Gromov-Witten potential
is a genus 2 meromorphic Siegel modular form of weight $2g-2$ for a
certain discrete subgroup $P^{*}_{N} \subset \Sp_{4}(\RR)$.

We also compute the weight 4 modular form whose $p$th Fourier
coefficient is given by the trace of the action of Frobenius on
$H^{3}_{\et }(\XtildeN ,{\QQ}_{\ell})$ for almost all prime
$p$. We observe that it is the unique weight 4 cusp form on 
$\Gamma_{0}(N)$.
\end{abstract}

\maketitle 



\section{Introduction: the geography, enumerative geometry, and
arithmetic of Calabi-Yau threefolds.}

In this paper, a \emph{Calabi-Yau threefold (CY3)} is a smooth,
projective threefold $X$ over $\CC$ with $K_{X}\cong \OO_{X}$ and
$H^{1}(X,\CC )=0.$ We are interested three aspects of a CY3
$X$. Namely, geography (the Hodge numbers of $X$), enumerative
geometry (curve counting on $X$), and arithmetic (for rigid $X$,
counting points over $\FF_{p}$ gives rise to a weight 4 modular form).


In this paper, we construct four new CY3s which are interesting from
all three points of view.

\subsection{Geography} The Hodge numbers of a CY3 $X$ are determined
by the two values $h^{1,1}(X)$ and $h^{2,1}(X)$ which have geometric
significance: $h^{1,1}(X)$ is the Picard number of $X$ and
$h^{2,1}(X)$ is the dimension of the space of deformations of $X$. The
\emph{geography problem for CY3s} asks which pairs of numbers
$(h^{1,1},h^{2,1})$ occur as the Hodge numbers of a CY3. There has
been some recent interest in the physics community in the existence of
CY3s with small Hodge numbers. Candelas, Constantin, and Mishra
\cite{Candelas-CY3-small-Hodge} list all known (at the time) CY3s with height
$h^{1,1}+h^{2,1}\leq 24$. We call a CY3 with height satisfying
\[
h^{1,1}+h^{2,1}\leq 6
\]
a \emph{nano-manifold}. We construct four new examples of rigid
nano-manifolds of height 4.

\begin{theorem}\label{thm: XtildeN exists}
There exist CY3s $\XtildeN$, $N\in \{5,6,8,9 \}$ with 
\[
h^{1,1}(\XtildeN )=4,\quad h^{2,1}(\XtildeN  ) = 0.
\]
There is an Abelian surface fibration $\XtildeN \to \PP^{1}$,
with four singular fibers, whose generic fiber is an Abelian surface
of Picard number three. $\Pic (\XtildeN )$ is spanned by the
class of the fiber and the three divisor classes of the generic
fiber. $\XtildeN$ has fundamental group of order $N$
(c.f. Theorem~\ref{thm: BlGamma(Xcon) is a resolution}).
\end{theorem}

\subsection{Enumerative Geometry}

The enumerative geometry of a CY3 $X$ is the study of curve counting
on $X$. There are several equivalent curve counting theories on a
CY3. We are particularly interested in \emph{Donaldson-Thomas (DT)}
invariants, \emph{Gromov-Witten (GW)} invariants, and
\emph{Gopakumar-Vafa (GV)} invariants.

\subsubsection{DT theory.}
As usual, we may define the DT invariants of a CY3 $X$
as the Behrend function weighted Euler characteristic of the Hilbert
scheme  
\[
\DT_{n,\beta}(X) = e(\Hilb^{n,\beta}(X),\nu ).
\]
Here $\Hilb^{n,\beta}(X)$ is the Hilbert scheme of subschemes of $X$
supported in the class $\beta \in H_{2}(X)$ and having holomorphic
Euler characteristic $n$, and $e(-,\nu )$ denotes topological Euler
characteristic weighted by Behrend's constructible function $\nu$
\cite{Behrend-micro}.

For a basis of divisor classes $D_{1},\dotsc ,D_{r}$, we define the
\emph{DT partition function} by
\[
Z_{X}(p,q_{1},\dotsc ,q_{r}) = \sum_{n,\beta} \DT_{n,\beta}(X)\,
(-p)^{n}\,q_{1}^{\beta \cdot D_{1}} \dotsb q_{r}^{\beta \cdot D_{r}} \, .
\] 

The nano-manifolds $\XtildeN$ are of \emph{banana type} which makes it
possible to compute the DT partition function of $\XtildeN$ for all
fiber curve classes. To be of banana type means that $\XtildeN
\to \PP^{1}$ is the compactification of a group scheme
$\XtildeN^{\circ}\to \PP^{1}$ by banana configurations.

A \emph{banana configuration} in a CY3 is a union of three smoothly
embedded $(-1,-1)$ rational curves $C_{1},C_{2},C_{3} $ whose
intersection consists of two points $\{p,q \}$. Moreover, the group of
the singular fiber where a banana configuration occurs is
$\CC^{*}\times \CC^{*}$ and it acts on the banana configuration and
its formal neighborhood (see \cite{Bryan-Banana}).

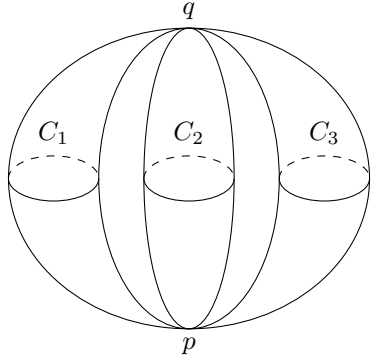
\begin{center}

\begin{tikzpicture}[xshift=5cm,
		    scale = 1.0
		    ]

\begin{scope}  
\draw (0,0) ellipse (2.4 and 2);
\draw (0,0) ellipse (0.6cm and 2cm);
\draw (0,0) ellipse (1.2cm and 2cm);
\draw (-0.6,0) arc(180:360:0.6 and 0.3);
\draw[dashed](-0.6,0) arc(180:0:0.6cm and 0.3cm);
\draw (1.2,0) arc(180:360:0.6cm and 0.3cm);
\draw[dashed](1.2,0) arc(180:0:0.6cm and 0.3cm);
\draw (-2.4,0) arc(180:360:0.6cm and 0.3cm);
\draw[dashed](-2.4,0) arc(180:0:0.6cm and 0.3cm);

\draw (0,-2) node[below] {$p$};
\draw(0,2) node[above] {$q$};
\draw(0,0.6) node {$C_{2}$};
\draw(-1.8,0.6) node {$C_{1}$};
\draw(1.8,0.6)node {$C_{3}$};
\end{scope}

\end{tikzpicture}

\end{center}

There is a basis of divisor classes for $\XtildeN$ given by $\Ftilde
,\Deltatilde, \Sprimetilde  ,\Stilde $ where $\Ftilde$ is the class of
the fiber, and $\Deltatilde,\Sprimetilde  ,\Stilde $ are classes which,
when restricted to a generic fiber, span the Picard group and have
intersection form
$
\left(
\begin{smallmatrix}
-2N^{2}&0&0\\
0&0&N\\
0&N&0
\end{smallmatrix}
 \right) $. We assign variables $z,y,q,Q$ to the divisors $\Ftilde
,\Deltatilde, \Sprimetilde ,\Stilde $ respectively and we compute the
DT partition function for fiber curve classes, i.e. we compute the
$z\to 0$ limit of $Z_{\XtildeN}(p,z,y,q,Q)$.

\begin{theorem}\label{thm: ZDT formula} Define positive integers $c(a,m)$  as
the Fourier coefficients of the ratio of theta functions
$\theta_{4}(q^{2} ,y)/\theta _{1}(q^{4} ,y)^{2}$, namely:
\begin{equation}\label{eqn: definition of the coefficientsc(a,m)}
\sum_{a=-1}^{\infty}\sum_{m\in \ZZ} c(a,m) q^{a}y^{m} =
\frac{\sum_{m\in \ZZ}q^{m^{2}}(-y)^{m}}{\left(\sum_{m\in \ZZ +\half}
q^{2m^{2}}(-y)^{m}\right)^{2}}. 
\end{equation}
Then the DT partition function of $\XtildeN$, for fiber
curve classes is given by
\[
Z_{\XtildeN}(p,y,q,Q) = \prod_{k\in \Theta_{N}} \prod_{m,r,s,t}
\left(1-p^{m}q^{Nr/k}Q^{ks}y^{Nt} \right)^{-c(4rs-t^{2},m)}
\]
where in the sum $m,r,s,t$ are integers with $r,s,r+s+t\geq 0$ and if
$r=s=t=0$ then $m>0$, and  $\Theta_{N}$ is the 4-tuple given by
\[
\Theta_{5}=(5,5,1,1),\quad \Theta_{6}=(6,3,2,1),\quad
\Theta_{8}=(4,4,2,2),\quad \Theta_{9}=(3,3,3,3). 
\]
\end{theorem}

\begin{remark}
The methods developed in \cite{Bryan-Banana} to compute the fiber
curve class DT partition function of the ordinary banana manifold
applies equally well to the banana nano-manifolds $\XtildeN$. The
primary difference is that in the ordinary banana manifold, the three
banana curves of any singular fiber span the fiber curve classes
whereas the classes of the banana curves in the singular fibers of
$\XtildeN$ are more complicated and vary from singular fiber to
singular fiber. For this reason, it is more convenient to express the
paritition function in terms of the divisor classes and it accounts
for the differences between the formula in this paper versus the one
in \cite{Bryan-Banana}. 
\end{remark}

\subsubsection{GW theory.}\label{subsection: GW theory}
We define the genus $g$ GW potential for fiber curve
classes by
\[
F_{g}^{\XtildeN}(Q,q,y) = \sum_{\begin{smallmatrix} \beta \in H_{2}(\XtildeN)\\
\pi_{*}\beta =0 \end{smallmatrix}
} \left\langle \, \, \right\rangle^{\XtildeN}_{g,\beta}\, \,  Q^{\beta \cdot \Stilde}q^{\beta \cdot \Sprimetilde}y^{\beta \cdot \Deltatilde}
\]
where $ \left\langle \, \, \right\rangle^{\XtildeN}_{g,\beta}$ is the
genus $g$ GW invariant of $\XtildeN$ in the class $\beta$.

The following is a consequence of Theorem~\ref{thm: ZDT formula}:

\begin{corollary}\label{cor: formula for F_g}
The genus $g$ GW potential of $\XtildeN$ is given by
\begin{equation}\label{eqn: GW potential of Xtilde}
F_{g}^{\XtildeN}(Q,q,y) = \sum_{k\in \Theta_{N}} F_{g}^{\ban}(Q^{k},q^{N/k},y^{N})
\end{equation}
where for $g\geq 2$, $F_{g}^{\ban}(Q,q,y)$ is a meromorphic Siegel
modular form of weight $2g-2$ with 
\[
Q=e^{2\pi i\sigma},q=e^{2\pi i\tau}, y=e^{2\pi iz}, \text{ for }
\left(\begin{smallmatrix} \tau &z\\ 
z&\sigma \end{smallmatrix} \right)\in \HH_{2}
\]
where $\HH_{2}$ is the genus 2 Siegel
upper half space. Namely, $F_{g}^{\ban}$ is the Maass lift of the
index 1, weight $2g-2$ Jacobi form 
\[
\alpha_{g}\cdot \phi_{-2,1}(q,y)\cdot E_{2g}(q)
\]
where $E_{2g}(q)$ is the weight $2g$ Eisenstein series, $\phi_{-2,1}(q,y)$
is the unique weak Jacobi form of weight -2 and index 1, and
$\alpha_{g}=\frac{|B_{2g}|}{2g(2g-2)!}$
(c.f. \cite[\S~A.4]{Bryan-Banana} for notation).
\end{corollary}

We show that (up to a change of variables) the GW potentials $F_{g}^{\XtildeN}$ are Siegel modular forms for a certain subgroup of $\Sp_{4}(\RR)$ as follows. For $N \in \{5, 6, 8, 9\}$ consider the group
\begin{equation}\label{eqn: defn of subgroup PN}
P_{N} = \Sp_{4}(\QQ) \cap
\begin{pmatrix}
\ZZ & \tfrac{N}{d_{N}} \, \ZZ & \tfrac{1}{d_{N}} \, \ZZ & \tfrac{1}{N} \, \ZZ \\[5pt]
\frac{1}{d_{N}} \, \ZZ & \ZZ & \frac{1}{N} \, \ZZ & \frac{1}{N d_{N}} \, \ZZ \\[5pt]
\frac{N}{d_{N}} \, \ZZ  & N \, \ZZ  & \ZZ & \frac{1}{d_{N}} \, \ZZ  \\[5pt]
N \, \ZZ & \frac{N^{2}}{d_{N}} \, \ZZ & \frac{N}{d_{N}} \, \ZZ & \ZZ 
\end{pmatrix}
\end{equation}
where $d_{N} = \gcd \Theta_{N}$. Explicitly, $d_{5} = 1$, $d_{6}=1$,
$d_{8} = 2$, and $d_{9}=3$. We additionally consider the involution  
\begin{equation}\label{eqn: defn of involution iN}
\iota_{N} = 
\begin{pmatrix}
0 & \sqrt{N} & 0 & 0 \\
\frac{1}{\sqrt{N}} & 0 & 0 & 0 \\
0  & 0  & 0 & \frac{1}{\sqrt{N}}  \\
0 & 0 & \sqrt{N} & 0 
\end{pmatrix}
\in \Sp_{4}(\RR)
\end{equation}
and observe that conjugation by $\iota_{N}$ preserves $P_{N}$. We can
therefore produce an index $2$ normal extension of $P_{N}$ 
\[
P^{*}_{N} = P_{N} \cup \iota_{N} P_{N}
\]
and prove the following
\begin{corollary}\label{cor: GW potential of Xtilde is a paramodular form}
Up to a change of variables (see Section \ref{subsection: GW
potentials of XtildeN}), the GW potentials $F_{g}^{\XtildeN}$ for $g
\geq 2$ are meromorphic Siegel modular forms of weight $2g-2$ for the
discrete subgroup $P^{*}_{N} \subset \Sp_{4}(\RR)$.  
\end{corollary}

The involution $\iota_{N}$ induces the transformation $(Q, q, y)
\mapsto (q^{\frac{1}{N}}, Q^{N}, y)$, which is a symmetry of
$F_{g}^{\XtildeN}$ after the change of variables. We note that this is
reminiscent of the behaviour of \emph{Siegel paramodular forms} where
in particular, invariance under $(Q, q, y) \mapsto (q^{\frac{1}{N}},
Q^{N}, y)$ is a key property \cite[\S~3.1]{Gritsenko99}.

%
%
%

\subsubsection{GV theory.}
The genus $g$, curve class $\beta$ GV invariant of a CY3 $X$ is an
integer valued curve counting invariant $n^{g}_{\beta}(X)$ which has
been given a geometric definition by Maulik and Toda
\cite{Maulik-Toda-Inventiones-2018}. The GV invariants are
conjecturally equivalent to the DT/GW invariants by the Gopakumar-Vafa
formula which can alternatively be used to give a (non-geometric)
definition. Either by assuming the Gopakumar-Vafa formula holds for
$\XtildeN$, or by using the non-geometric definition, we may use our
computation of the DT partition function to compute the (fiber class)
GV invariants of $\XtildeN$.

The curve classes of a smooth fiber of $\XtildeN \to \PP^{1}$ generate
(over $\QQ$) the fiber curve classes of $\XtildeN$. Consequently, the
fiber curve classes inherit a quadratic form $||\cdot ||$ coming from
the intersection pairing on a smooth fiber.

\begin{proposition}\label{prop: formula for GV invariants of XN}
Assuming that the Gopakumar-Vafa formula holds for $\XtildeN$, the
Gopakumar-Vafa invariants of $\XtildeN$ in an effective fiber class
$\beta$ with $2||\beta ||=a$, are given by
\[
n^{g}_{\beta}(\XtildeN ) = \epsilon_{N}(\beta
)\, n^{g}_{a}
\]
where
\[
\epsilon_{N}(\beta ) =\sum_{k\in \Theta_{N}} \epsilon_{N,k}(\beta
),\quad \quad \epsilon_{N,k}(\beta ) = \begin{cases}
1&\text{if $k\divides (\beta \cdot \Stilde)$ and $\frac{N}{k}\divides (\beta \cdot \Sprimetilde)$}\\
0&\text{otherwise,}
\end{cases}
\]
and the integers $n^{g}_{a}$ are given by
\[
\sum_{a=-1}^{\infty}\sum_{g=0}^{\infty} n_{a}^{g}\,
(y^{\half}+y^{-\half})^{2g}q^{a+1} =  \prod_{n=1}^{\infty}
\frac{(1+yq^{2n-1})(1+y^{-1}q^{2n-1})
(1-q^{2n})}{(1+yq^{4n})^{2}(1+y^{-1}q^{4n})^{2}(1-q^{4n})^{2} }.
\]
\end{proposition}

We remark that in the usual banana manifold, the GV invariants in an
effective fiber class $\beta$ only depend on the quadratic form induced by
the generic fiber whereas for $\XtildeN$, they also depend mildly on
divisibility conditions through the number $\epsilon_{N}(\beta )\in
\{1,2,3,4 \}$. 

\subsection{Arithmetic}

Let $X$ be a rigid CY3 defined over $\QQ$. Then there is a continuous
system of two dimensional representations of
$G_{\QQ}=\operatorname{Gal}(\overline{\QQ},\QQ )$ on
$H_{\et}^{3}(X_{\overline{\QQ}},\QQ_{l})$. It was shown by
Gouv\^{e}a-Yui \cite{Gouvea-Yui-rigidCY3s-are-modular} and Dieulefait
\cite{Dieulefait-Modularity-of-CY3s} using the proof of Serre's
conjectures by Khare-Wintenberger
\cite{Khare-WintenbergerI,Khare-WintenbergerII} that this system of
Galois representations is \emph{modular}. In particular, there exists a
weight 4 modular cusp form
\[
f_{X}(q) = \sum_{n=1}^{\infty} a_{n}q^{n}
\] 
uniquely characterized by the condition that 
\begin{equation}\label{eqn: ap=trac(Frobp)}
a_{p} =\Tr (\Frob_{p} \divides H^{3}_{\et}(X_{\overline{\FF}_{p}},\QQ_{l}))
\end{equation}
for almost all primes $p$. Moreover, $\Tr (\Frob_{p})$ can be
determined by the Lefschetz fixed point formula in terms of $\#
X_{\FF_{p}}$ and so the arithmetic of counting points in $X$ over
$\FF_{p}$ gives rise to a modular form $f_{X}(q)$\footnote{Modularity
for rigid CY3s defined over $\QQ$ is
analogous to the famous modularity theorem of Wiles-Taylor for
elliptic curves over $\QQ$. In that case the associated modular form
$f_{E}(q)$ is weight 2. }.

\begin{theorem}\label{thm: theorem on weight four cusp forms}
The rigid CY3 $\XtildeN$ is defined over $\QQ$. The corresponding
modular form $f_{\XtildeN}(q)$ is the unique weight $4$ cusp form on
the congruence subgroup $\Gamma_{0}(N)$. It can be expressed as the
eta product
\[
f_{\XtildeN}(q) = \prod_{k \in \Theta_{N}} \eta(q^{k})^{2}
\]
where $\eta(q) = q^{1/24} \prod_{n \geq 1}(1-q^{n})$ is the Dedekind
eta function, and $\Theta_{N}$ is as in Theorem \ref{thm: ZDT
formula}.
\end{theorem}

The proof of this theorem, given in Appendix~\ref{sec: reduction to
Verills case}, reduces to our case to a previously known case. Namely,
we show that 
\[
f_{\XtildeN}(q)  = f_{X_{N}^{\Ver}}(q)
\]
where $X_{N}^{\Ver}$ are the rigid CY3s studied by Verrill
\cite[Appendix]{Yui-with-Verrill-appendix} who  carried
out the requisite point counts to determine $f_{X_{N}^{\Ver}}(q)$.

The banana nano-manifold $\XtildeN\to \PP^{1}$ has four singular
fibers over points $p_{1},\dotsc ,p_{4}\in \PP^{1}$ and thus there is
an associated elliptic curve $E_{N}$ given by the double branched
cover of $\PP^{1}$ ramified over $p_{1},\dotsc ,p_{4}$. We have
observed the following curious proposition which we will prove in
Appendix~\ref{sec: reduction to Verills case}.

\begin{proposition}\label{prop: f_E(q)^2 = f_XN(q^2)}
There is a model for $E_{N}$ defined over $\QQ$ such that the
associated weight 2 modular form $f_{E_{N}}(q)$ is given by
\[
f_{E_{N}}(q)=  \prod_{k \in \Theta_{N}} \eta(q^{2k})
\]
so that $f_{E_{N}}(q)^{2} = f_{\XtildeN}(q^{2})$.
\end{proposition}

There is also elliptic curve associated to any rigid CY3 $X$, namely the
intermediate Jacobian $J(X)$ which is given by
\[
0\to H_{3}(X,\ZZ) \xrightarrow{i} H^{3,0}(X)^{*}\to J(X)\to 0
\]
where the embedding $i$ is given by integrating the Calabi-Yau form
over 3-cycles. This raises the natural 
\begin{question}\label{question: is EXN = J(XN)?}
Is $E_{N}\cong J(\XtildeN )$ ?
\end{question}
More generally, we ask the following
\begin{question}\label{question: when does f_X(q^2) = f_E(q)^2?}
Let $X$ be a rigid CY3 defined over $\QQ$ and let $E=J(X)$ be its
intermediate Jacobian. When does there exist a model for $E$ defined
over $\QQ$ such that the equation
\[
f_{X}(q^{2}) = f_{E}(q)^{2} 
\]
holds?
\end{question}

\subsection{Plan of the paper.}\label{subsec: outline of paper}

In Section~\ref{sec: the construction of XN} we give our construction
of $\XtildeN$ which is given as a free quotient of a crepant
resolution of a certain fiber product of extremal rational elliptic
surfaces. In Section~\ref{sec: local toric geometry} we use toric
geometry to prove some key facts about our construction. In Section
~\ref{sec: Divs curves and int numbers on XN}, we find a basis for
curve classes and for divisors on $\XtildeN$ and compute their
intersection numbers. In Section~\ref{sec: computation of DT/GW/GV of
XN}, we use the techniques of \cite{Bryan-Banana} to compute the DT
invariants of $\XtildeN$ and we use that computation to get the GW and
GV invariants. In Appendix~\ref{sec: reduction to Verills case}, we
prove Theorem~\ref{thm: theorem on weight four cusp forms} by use
\'etale cohomology techniques to reduce our computation to a related
computation done by Verrill \cite{Yui-with-Verrill-appendix}.

\section{The Construction of $\XtildeN$}\label{sec: the construction
of XN}

Let $S_{N}\to \PP^{1}$ be one of the four unique rational elliptic
surfaces having four singular fibers of type
$(I_{k_{1}},I_{k_{2}},I_{k_{3}},I_{k_{4}})$ where $\Theta_{N} =
(k_{1},k_{2},k_{3},k_{4})$ is given below:

\bigskip

\setlength{\tabcolsep}{1.0em} 
{\renewcommand{\arraystretch}{1.2}

\begin{center}

\begin{tabular}{llll}
$N$ & Singularity Type $\Theta_{N}$ & Mordell-Weil Group $G_{N}$ & Modular group $\Gamma_{N}$\\
\hline 
$5$ & $(5,5,1,1)$ & $\ZZ_{5}$     & $\Gamma_{1}(5)$\\
$6$ & $(6,3,2,1)$ & $\ZZ_{6}$  & $\Gamma_{1}(6) $\\
$8$ & $(4,4,2,2)$ & $\ZZ_{4} \times \ZZ_{2}$   & $\Gamma_{1}(4)\cap \Gamma (2) $\\
$9$ & $(3,3,3,3)$ & $\ZZ_{3} \times \ZZ_{3}$  & $\Gamma (3)$\\
\end{tabular}
\end{center}
\bigskip

The surface $S_{N}$ is a compactification of the universal elliptic
curve over the modular curve 
\[
\PP^{1}-\{\text{4 points} \}\cong \HH /\Gamma_{N}
\]
where $\HH$ is the upper half-plane and $\Gamma_{N}\subset SL_{2}(\ZZ
)$ is the congruence subgroup listed above. $S_{N}$ admits an action
of $G_{N}$, the associated Mordell-Weil group of sections listed
above. The modular curve can be regarded as parameterizing elliptic
curves $E$ equipped with an injective homomorphism $G_{N}\to E$. Note
that the order of $G_{N}$ is $N$. $S_{N}$ is an \emph{extremal
elliptic surface}: it admits no deformations preserving the
singularity type
\cite{Miranda-Persson-Extremal,Beauville-Elliptic-surface-with-4-singularfibers}.

We will henceforth often suppress the $N$ subscript of $S_{N}$ and
$G_{N}$. Let
\[
S'_{\sing} = S/G,\quad S'\to S'_{\sing}
\]
be the quotient and its minimal resolution. The action of $G$ on $S$ is
free away from the nodes of the singular fibers and it is easy to see
that the stabilizer of the nodes in an $I_{k}$ fiber is
$\ZZ_{N/k}$. It then follows that $S'$ is isomorphic to $S$ where the
isomorphism is induced from an automorphism of the base $\PP^{1}$
which reverses the order of the singular fibers. Indeed, the quotient
of an $I_{k}$ fiber in $S$ is an irreducible nodal rational curve in
$S'_{\sing}$ with a $\ZZ_{N/k}$ quotient singularity at the node. We
denote this type of fiber by $I_{1}^{\ZZ_{N/k}}$.

Let 
\begin{align*}
X_{\sing} & = S\times_{\PP^{1}} S'_{\sing},\\
X_{\con} & = S\times_{\PP^{1}} S'.
\end{align*}
It follows from the previous analysis that 
\[
X_{\sing}\to \PP^{1} \text{ and } X_{\con} \to \PP^{1} 
\]
are Abelian surface fibrations with four singular fibers of type
\[
I_{k}\times I_{1}^{\ZZ_{N/k}}  \text{ and } I_{k} \times I_{N/k}
\]
respectively for $k\in \Theta_{N}$.

The threefolds $X_{\sing}$ and $X_{\con}$ are singular
CY3s. $X_{\con}$ has $4N$ singularities occurring at the points
\[
n_{i}\times n'_{j}\in I_{k}\times I_{N/k} 
\]
where $n_{i}\in I_{k}$ and $n'_{j}\in I_{N/k}$ are nodes. The
singularities of $X_{\con}$ are all conifold singularities whereas the
singularities of $X_{\sing}$ are more complicated but are locally
hypersurface singularities given by the equation $\{ab=x^{l}y^{l} \}$
where $l=N/k$.

We illustrate $X_{\sing}$ and $X_{\con}$ below in the case of $N=6$:

\bigskip

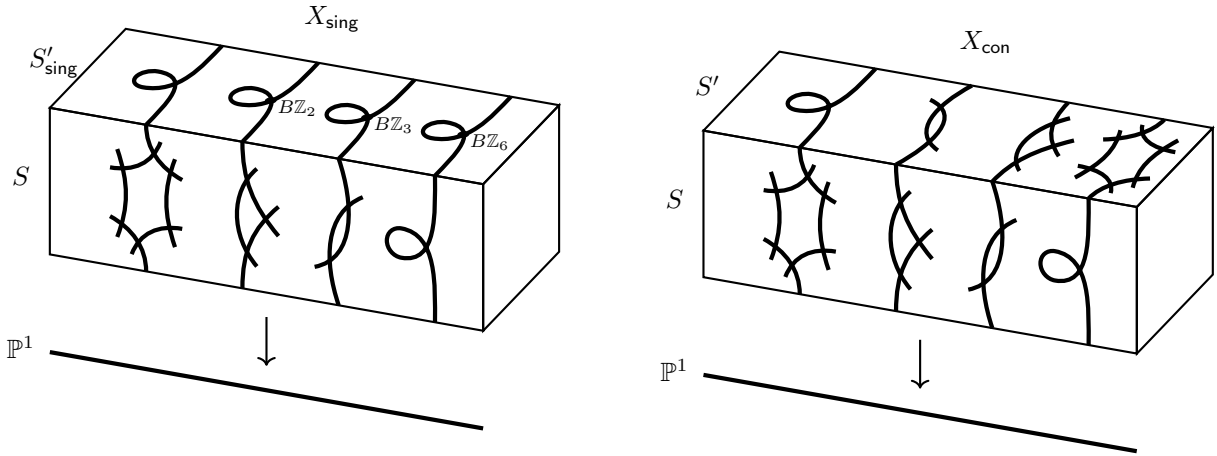
\begin{center}


\begin{tikzpicture}
[x={({cos(10)},{-sin(10)},0)},z={({-sin(40)},{-cos(40)},-0.3)},scale=0.65]

\pgfmathsetmacro{\rshift}{12}
\pgfmathsetmacro{\ushift}{-3}
\begin{scope}[canvas is xz plane at y=3] 
\draw [thick] (0+\rshift,0+\ushift) rectangle (9+\rshift,3+\ushift);

\draw [ultra thick,black] 
                    (2 +0 +\rshift ,0 +\ushift  ) 
to [out=90,in=0]    (1.3  +0 +\rshift,1.8+\ushift ) 
to [out=180,in=90]  (1  +0 +\rshift ,1.5 +\ushift) 
to [out=270,in=180] (1.3  +0+\rshift ,1.2+\ushift ) 
to [out=0,in=270]   (2   +0 +\rshift ,3 +\ushift  );

\draw [ultra thick,black] 
(1.5  +2 +\rshift,0.7 +\ushift) 
to [out=20,in=290]   (2   +2 +\rshift  ,3+\ushift   );
\draw [ultra thick,black] 
                    (2 +2  +\rshift ,0 +\ushift  ) 
to [out=110,in=210]    (2.5  +2 +\rshift,2.3+\ushift );

\draw [ultra thick,black] (2+4+\rshift,3+\ushift) to[out=-90,in=135] (2.75+4+\rshift,1.2+\ushift);
\draw [ultra thick,black] (2+4+\rshift,0+\ushift) to[out=90,in=-135] (2.75+4+\rshift,1.8+\ushift);
\draw [ultra thick,black] (2.3+4+\rshift,0.5+\ushift) to[out=135,in=-135] (2.3+4+\rshift,2.5+\ushift);

\draw [ultra thick,black] (2+6+\rshift,0+\ushift) to[out=90,in=-20] (1.25+6+\rshift,1.0+\ushift);
\draw [ultra thick,black] (1.4+6+\rshift,0.6+\ushift) to[out=70,in=-70] (1.4+6+\rshift,2.4+\ushift);
\draw [ultra thick,black] (2+6+\rshift,3+\ushift) to[out=-90,in=160] (2.75+6+\rshift,2+\ushift);
\draw [ultra thick,black] (2.6+6+\rshift,0.6+\ushift) to[out=110,in=-110] (2.6+6+\rshift,2.4+\ushift);
\draw [ultra thick,black] (1.7+6+\rshift,0.3+\ushift) to[out=70,in=180] (2.75+6+\rshift,1.0+\ushift);
\draw [ultra thick,black] (2.3+6+\rshift,2.7+\ushift) to[out=-110,in=0] (1.25+6+\rshift,2+\ushift);

\node at (-0.7+\rshift,1.5+\ushift) {$S'$};
\node at (3.5+\rshift,-1.5+\ushift) {$X_{\con}$};
\end{scope}

\begin{scope}[canvas is xy plane at z=3+\ushift] 
\draw [thick] (0+\rshift,0) rectangle (9+\rshift,3);
\draw [ultra thick] (0+\rshift,-2 ) to (9+\rshift,-2 );
\node  at (-0.6+\rshift,-2 ) {$\PP^{1}$};
\draw [->,thick] (4.5+\rshift,-0.5 )--(4.5+\rshift,-1.5 );

\draw [ultra thick,black] (2+\rshift,0 ) to[out=90,in=-20] (1.25+\rshift,1.0 );
\draw [ultra thick,black] (1.4+\rshift,0.6 ) to[out=70,in=-70] (1.4+\rshift,2.4 );
\draw [ultra thick,black] (2+\rshift,3 ) to[out=-90,in=160] (2.75+\rshift,2 );
\draw [ultra thick,black] (2.6+\rshift,0.6 ) to[out=110,in=-110] (2.6+\rshift,2.4 );
\draw [ultra thick,black] (1.7+\rshift,0.3 ) to[out=70,in=180] (2.75+\rshift,1.0 );
\draw [ultra thick,black] (2.3+\rshift,2.7 ) to[out=-110,in=0] (1.25+\rshift,2 );

\draw [ultra thick,black] (2+2+\rshift,3 ) to[out=-90,in=135] (2.75+2+\rshift,1.2 );
\draw [ultra thick,black] (2+2+\rshift,0 ) to[out=90,in=-135] (2.75+2+\rshift,1.8 );
\draw [ultra thick,black] (2.3+2+\rshift,0.5 ) to[out=135,in=-135] (2.3+2+\rshift,2.5 );

\draw [ultra thick,black] 
(1.5  +4 +\rshift,0.7  ) 
to [out=20,in=290]   (2   +4 +\rshift ,3    );
\draw [ultra thick,black] 
                    (2 +4 +\rshift ,0    ) 
to [out=110,in=210]    (2.5  +4+\rshift,2.3  ); 

\draw [ultra thick,black] 
                    (2 +6 +\rshift ,0    ) 
to [out=90,in=0]    (1.3  +6+\rshift ,1.8  ) 
to [out=180,in=90]  (1  +6 +\rshift ,1.5  ) 
to [out=270,in=180] (1.3  +6+\rshift ,1.2  ) 
to [out=0,in=270]   (2   +6 +\rshift ,3   );

\node at (-0.6+\rshift,1.5 ) {$S$};
\end{scope}

\begin{scope}[canvas is yz plane at x=9+\rshift] 
\draw [thick] (0,0+\ushift) rectangle (3,3+\ushift);

\end{scope}

\begin{scope}[canvas is xz plane at y=3] 
\draw [thick] (0,0) rectangle (9,3);


\foreach \x in {0,2,4,6}
\draw [ultra thick,black] 
                    (2 +\x   ,0   ) 
to [out=90,in=0]    (1.3  +\x ,1.8 ) 
to [out=180,in=90]  (1  +\x   ,1.5 ) 
to [out=270,in=180] (1.3  +\x ,1.2 ) 
to [out=0,in=270]   (2   +\x  ,3   );

\foreach \x\y in {2/2,4/3,6/6}
{\node at (2.4+\x ,1.6) {$\scriptstyle{B\ZZ_{\y}}$};
\draw [fill] (1.8+\x,1.5) circle [radius=0.1];}
\node at (-0.7,1.5) {$S'_{\sing}$};
\node at (3.5,-1.5) {$X_{\sing}$};
\end{scope}

\begin{scope}[canvas is xy plane at z=3] 
\draw [thick] (0,0) rectangle (9,3);
\draw [ultra thick] (0,-2) to (9,-2);
\node  at (-0.6,-2) {$\PP^{1}$};
\draw [->,thick] (4.5,-0.5)--(4.5,-1.5);


\draw [ultra thick,black] (2,0) to[out=90,in=-20] (1.25,1.0);
\draw [ultra thick,black] (1.4,0.6) to[out=70,in=-70] (1.4,2.4);
\draw [ultra thick,black] (2,3) to[out=-90,in=160] (2.75,2);
\draw [ultra thick,black] (2.6,0.6) to[out=110,in=-110] (2.6,2.4);
\draw [ultra thick,black] (1.7,0.3) to[out=70,in=180] (2.75,1.0);
\draw [ultra thick,black] (2.3,2.7) to[out=-110,in=0] (1.25,2);

\draw [ultra thick,black] (2+2,3) to[out=-90,in=135] (2.75+2,1.2);
\draw [ultra thick,black] (2+2,0) to[out=90,in=-135] (2.75+2,1.8);
\draw [ultra thick,black] (2.3+2,0.5) to[out=135,in=-135] (2.3+2,2.5);

\draw [ultra thick,black] 
(1.5  +4 ,0.7 ) 
to [out=20,in=290]   (2   +4  ,3   );
\draw [ultra thick,black] 
                    (2 +4  ,0   ) 
to [out=110,in=210]    (2.5  +4,2.3 ); 

\draw [ultra thick,black] 
                    (2 +6  ,0   ) 
to [out=90,in=0]    (1.3  +6 ,1.8 ) 
to [out=180,in=90]  (1  +6  ,1.5 ) 
to [out=270,in=180] (1.3  +6 ,1.2 ) 
to [out=0,in=270]   (2   +6  ,3   );

\node at (-0.6,1.5) {$S$};
\end{scope}

\begin{scope}[canvas is yz plane at x=9] 
\draw [thick] (0,0) rectangle (3,3);

\end{scope}

\end{tikzpicture}

\end{center}
\bigskip


It follows from the previous analysis of the stabilizers of $G$ that
the diagonal copy of $G$ in $G\times G$ acts freely on
$S\times_{\PP^{1}}S'$, transitively permuting the $N$ conifold points
in each singular fiber.

One of our main technical results is that there is a projective
conifold resolution of $X_{\con}$ given by blowing up an explicit
(non-Cartier) Weil divisor passing through all the conifold
points. Namely, let
\[
\Gamma \subset X_{\con}
\]
be the proper transform of $\gamma \subset X_{\sing}$ where $\gamma$
is the graph of the quotient map
\[
f:S\to S'_{\sing} .
\]

\begin{theorem}\label{thm: BlGamma(Xcon) is a resolution}
Let $\Gamma \subset X_{\con}$ be as above and let
\[
X_{N} = \Bl_{\Gamma}(X_{\con})
\]
be the blow-up of $X_{\con}$ along $\Gamma$. Then $X_{N}$ is a
non-singular CY3 with $h^{2,1}(X_{N})=0$ and
$h^{1,1}(X_{N})=4N$. Moreover the quotient of $X_{N}$ by the diagonal
copy of $G_{N}$
\[
\XtildeN =X_{N}/G_{N}
\]
is a CY3 with $h^{2,1}(\XtildeN )=0$ and $h^{1,1}(\XtildeN )=4$. 
\end{theorem}

\begin{proof}
The most difficult part of the proof of the above theorem is showing
that blowing up $\Gamma$ yields a conifold resolution. We defer that
to the next section.

Assuming the conifold resolution exists, we compute the Hodge numbers
as follows. First we show that $X_{N}$ is rigid. For any CY3 given as
a conifold resolution $Z\to Z_{\con}$, the deformations of $Z$ are
given by the deformations of $Z_{\con}$ which do not smooth any of the
singularities \cite[\S~3.1]{Morrison-LookingGlass}. In the case of
$X_{\con}=S\times_{\PP^{1}}S'$, all deformations arise from
deformations of $S$, $S'$, or from composing the map $S\to \PP^{1}$
with a Mobi\"us transformation of the base \cite{Schoen-88}. But since
$S$ and $S'$ are extremal elliptic surfaces, any such deformation
results in smoothing one or more of the conifold
singularities. Therefore, $X_{N}$ is rigid and $h^{2,1}(X_{N})=0$.

The topological Euler characteristic of $X_{\con}$ can be computed as
follows. Since the generic fibers of $X_{\con}\to \PP^{1}$ are Abelian
surfaces and have Euler characteristic 0, the Euler characteristic of
$X_{\con}$ is equal to the sum of the Euler characteristics of the
singular fibers. Then since the singular fibers admit $\CC^{*}\times
\CC^{*}$ action whose only fixed points are the singularities, we find
\begin{align*}
e(X_{\con}) &= \text{number of conifold singularities }\\
&= 4N.
\end{align*}
The conifold resolution $X_{N}\to X_{\con}$ replaces each conifold
singularity with a $\PP^{1}$ and so we find that $e(X_{N})=8N$. Then
since for any CY3 $X$, $e(X) = 2h^{1,1}(X)-2h^{2,1}(X)$ we see that
$h^{1,1}(X_{N})=4N$ as was asserted. Finally since $G_{N}$ acts
freely, the Euler characteristic of $\XtildeN$ is 8 and subsequently,
$h^{1,1}(\XtildeN )=4$.

\end{proof}

\subsection{Schoen nano-manifolds}
We briefly digress to discuss some nano-manifolds closely related to
$\XtildeN$. Consider
\[
X^{\Sch}_{N} = S_{N} \times_{\phi} S'_{N}
\]
where the notation means that we take the fiber product of $S_{N}\to
\PP^{1}$ with the composition $S'_{N}\to \PP^{1}\xrightarrow{\phi
}\PP^{1}$ where $\phi:\PP^{1}\to \PP^{1}$ is a generic M\"obius
transformation.

The fiber product of any two rational elliptic surfaces such that the
singular fibers of each surface do not coincide (such as
$X^{\Sch}_{N})$ is called a Schoen manifold and is a smooth CY3 with
$h^{1,1}=h^{2,1}=19$. The 19 deformations are given by the 8
deformations of each rational elliptic surface along with the 3
dimensional family of M\"obius transformations of the base.

Consider the quotient of $X^{\Sch}_{N}$ by the free action of the
diagonal $G_{N}\subset G_{N}\times G_{N}$:
\[
\XtildeN^{\Sch} = X_{N}^{\Sch} / G_{N}.
\]

\begin{proposition}\label{prop: Nano Schoens have h11=h21=3}
$\XtildeN^{\Sch}$ is a nano-manifold of height 6 with
\[
h^{1,1}(\XtildeN^{\Sch})=h^{2,1}(\XtildeN^{\Sch})=3.
\]
Moreover, $\XtildeN^{\Sch}$ and $\XtildeN$ are related by a conifold
transition. 
\end{proposition}
\begin{proof}
Since $S$ and $S'$ are extremal, the only deformations of
$X^{\Sch}_{N}$ which preserve the $G_{N}$ action are the deformations
of $\phi$. Therefore $h^{2,1}(\XtildeN^{\Sch})=3$ and since 
\[
e(X^{\Sch}_{N}) = e(\XtildeN^{\Sch}) = 0
\]
we have $h^{1,1}(\XtildeN^{\Sch})=3$. Under the deformation taking a
generic $\phi$ to the identity, $\XtildeN^{\Sch}$ deforms to
$X_{\con}/G_{N}$ which has $\XtildeN$ as a conifold resolution.
\end{proof}

The enumerative geometry of the Schoen nano-manifolds
$\XtildeN^{\Sch}$ is expected to be interesting. On one hand, it is
related to the enumerative geometry of $\XtildeN$ by the usual
conifold transition formulas. On the other hand, the enumerative
geometry of $\XtildeN^{\Sch}$ should be related to the enumerative
geometry of the CHL model
\[
(K3\times E)/G_{N}
\]
where $G_{N}$ acts symplectically on $K3$ and by translation on the
elliptic curve $E$. The DT invariants of these models were studied in
\cite{Bryan-Oberdieck-CHL}.\footnote{The argument connecting these two
models is a folklore degeneration argument that we learned from
G. Oberdieck.}

\section{Local Toric Geometry, Resolutions, and
Intersections}\label{sec: local toric geometry}

In this section, we use toric geometry to analyze $\Gamma \subset
X_{\con}$ in a formal neighborhood of the singular points and prove
(among other things) that $X_{N}=\Bl_{\Gamma}X_{\con}$ is a conifold
resolution of $X_{\con}$.

Since $\Gamma \subset X_{\con}$ is a divisor, it is Cartier away from
the $4N$ conifold singularities and hence the blow-up does nothing
away from the singularities. Thus understanding
$X_{N}=\Bl_{\Gamma}(X_{\con})$ reduces to the local problem of
understanding $\Gamma$ in a neighborhood of each singular point. This
is still difficult since $\Gamma$ is defined as the proper transform
of $\gamma \subset X_{\sing}$ and is hence given by a closure (for
example, it turns out that $\Gamma$ is necessarily singular and
non-normal although we will not need to prove that here).

For each $k\in \Theta_{N}$ let 
\[
l=N/k
\]
so that the corresponding
singular fibers in $X_{\sing}$ and $X_{\con}$ are of type $I_{k}\times
I_{1}^{\ZZ_{l}}$ and $I_{k}\times I_{l}$ respectively. At the node in
$I_{1}^{\ZZ_{l}}$, the surface $S'_{\sing}$ is formally locally
modelled on 
\begin{align*}
S_{\sing}^{\loc} &\cong \AA^{2}/\ZZ_{l}\\
 &\cong \Spec \left(\CC [x,y]^{\ZZ_{l}} \right)\\
 &\cong \Spec \left(\CC [x^{l},y^{l},xy] \right)\\
 &\cong \Spec \left(\CC [a,b,c]/(ab-c^{l}) \right)\\
\end{align*}
and the map to the base is locally modelled on the map
\[
S_{\sing}^{\loc}\to \AA^{1}
\]
given by $(x,y)\mapsto xy$, i.e. $(a,b,c)\mapsto c$.

Let 
\[
S^{\loc}\to S^{\loc}_{\sing}
\]
be the minimal resolution. Note that the exceptional fiber is a chain
of $(l-1)$ $\PP^{1}$s.

The local model of $X_{\sing}$ at a singular point in the $I_{k}\times
I_{1}^{\ZZ_{l}}$ fiber is then
\begin{align*}
X^{\loc}_{\sing} &= \AA^{2}\times_{\AA^{1}} S_{\sing}^{\loc}\\
&=\Spec \left(\CC [x,y,a,b,c]/(ab-c^{l},xy-c) \right)\\
&=\Spec \left(\CC [x,y,a,b,]/(ab-x^{l}y^{l}) \right)\\
\end{align*}
and $X_{\con}\to X_{\sing}$ is locally modelled on 
\[
X^{\loc}_{\con}\to X^{\loc}_{\sing}\quad \text{ where }\quad  X^{\loc}_{\con} =
\AA^{2}\times_{\AA^{1}} S^{\loc}.
\]

We note that graph of the quotient map $\AA^{2}\to S^{\loc}_{\sing}$
is the Weil divisor $\gamma \subset X^{\loc}_{\sing}$ with ideal
\[
(a-x^{l},b-y^{l})\subset \CC [x,y,a,b]/(ab-x^{l}y^{l}).
\]

The threefolds $X^{\loc}_{\sing}$ and $X^{\loc}_{\con}$ are both
toric, the former is an affine toric variety with 1 singularity and
the latter is a quasi-projective toric variety with $l$ conifold
singularities.

The toric threefolds $X^{\loc}_{\sing}$ and $X^{\loc}_{\con}$
determine fans ($F_{\sing}$ and $F_{\con}$ respectively) in $\RR^{3}$
whose cones are generated by vertices in the $z=1$ plane in
$\RR^{3}$. For $X^{\loc}_{\sing}$, the fan is the cone over the
rectangle with vertices $(0,0,1)$, $(0,l,1)$, $(1,0,1)$, and
$(1,l,1)$. For $X^{\loc}_{\con}$, the fan is the union of the cones
over the $l$ squares with vertices $(0,i-1,1)$, $(0,i,1)$,
$(1,i-1,1)$, and $(1,i,1)$ for $i=1,\dotsc ,l$. We illustrate this
below for the case of $l=3$:

\begin{center}
\tikz[scale=0.4]{

\draw[gray!80!black, line width=.5pt, fill=gray!80!black, line
width=.5pt, fill opacity=.1, text
opacity=1](0,0)--++(0,18)--++(6,0)--++(0,-18)--cycle; 

\node[] at (-1.5,0){$(0,0,1)$};
\node[] at (-1.5,18){$(0,3,1)$};
\node[] at (7.5,0){$(1,0,1)$};
\node[] at (7.5,18){$(1,3,1)$};

\draw [red, thick](3,-2)--+(0,22);
\draw [red, thick](-2,9)--+(10,0);

\node [red] at (3,-3) {$b$};
\node [red] at (3,21) {$a$};
\node [red] at (-3,9) {$y$};
\node [red] at (9,9) {$x$};

\draw[very thick, orange, opacity=0.8] (-2,3) to[out=60, in=120]
++(2.5,3) to[out=-60, in=-120] ++(2.5,3); 

\draw[very thick, orange, opacity=0.8] (3,9) to[out=60, in=120]
++(2.5,3) to[out=-60, in=-120] ++(2.5,3); 

\node [orange] at (-3,5) {$b=y^{3}$};
\node [orange] at (9,13) {$a=x^{3}$};
\node [orange] at (8,15.5) {$\gamma $};
\node [orange] at (-2,2.5) {$\gamma $};

\begin{scope}[xshift=18cm,yshift=0cm]
\draw[gray!80!black, line width=.5pt, fill=gray!80!black, line
width=.5pt, fill opacity=.1, text
opacity=1](0,0)--++(0,18)--++(6,0)--++(0,-18)--cycle; 
\draw [gray!80!black, line width=0.5pt] (0,6)--(6,6);
\draw [gray!80!black, line width=0.5pt] (0,12)--(6,12);

\node[] at (-1.5,0){$(0,0,1)$};
\node[] at (-1.5,6){$(0,1,1)$};
\node[] at (-1.5,12){$(0,2,1)$};
\node[] at (-1.5,18){$(0,3,1)$};
\node[] at (7.5,0){$(1,0,1)$};
\node[] at (7.5,6){$(1,1,1)$};
\node[] at (7.5,12){$(1,2,1)$};
\node[] at (7.5,18){$(1,3,1)$};

\draw [red, thick](3,-2)--+(0,22);
\draw [red, thick](-2,3)--+(10,0);
\draw [red, thick](-2,9)--+(10,0);
\draw [red, thick](-2,15)--+(10,0);

\draw[very thick, decorate,decoration={coil,aspect=0, segment length=5pt}, orange,opacity=0.8] (3,3) --+(0,12);

\draw[very thick, orange, opacity=0.8] (0,-1) to[out=60, in=120] ++(1.5,2) to[out=-60, in=-135] ++(1.5,2);

\draw[very thick, orange, opacity=0.8] (3,15) to[out=60, in=120] ++(1.5,2) to[out=-60, in=-135] ++(1.5,2);

\node [orange] at (0,-1.5) {$\Gamma $};
\node [orange] at (6,19.5) {$\Gamma $};

\end{scope}

}

The fans $F_{\sing}$ and $F_{\con}$ with the divisors $\gamma$ and
$\Gamma $.
\end{center}
\bigskip

In the above pictures (which live in the $z=1$ plane), the cones of
the fans are given by the cones over the grey polygons. The dual
polytope is depicted in red and corresponds to the torus invariant
points, curves, and divisors\footnote{The faces of the dual polytope
should be perpendicular to the rays in the fan, so the red polytope
does not lie in the $z=1$ plane. We depict a projection of the dual
polytope to the plane.}. For example, the plane perpendicular to the
ray $(1,3,1)$ in the left picture corresponds to the torus invariant
(Weil) divisor given by the ideal
\[
(y,b)\subset \CC [x,y,a,b]/(ab-x^{3}y^{3}).
\]
This divisor is a copy of $\AA^{2}$ with coordinates $x$ and $a$ and
the intersection of $\gamma$ with it is given by the curve $a=x^{3}$
which we draw schematically in orange. The proper transform of $\gamma
\subset X_{\sing}^{\loc}$, namely $\Gamma \subset X^{\loc}_{\con}$,
can intersect the exceptional curves of $X^{\loc}_{\con}\to
X^{\loc}_{\sing}$ in complicated ways, which we depict in the right
hand picture with a squiggly orange curve.

To deal with the potential complications of $\Gamma$ and its blowup,
we find a family of divisors that interpolates between $\Gamma$ and a
toric divisor. Let $D_{\epsilon}\subset X^{\loc}_{\sing}$ be the Weil
divisor given by the ideal 
\[
(y^{l} - \epsilon b,a-\epsilon x^{l})\subset \CC [x,y,a,b]/(ab-x^{l}y^{l}).
\]
Then $D_{1}=\gamma$ and $D_{0}$ is the torus invariant divisor with
ideal $(y^{l},a)$. $D_{0}$ has support on the torus invariant divisor
with ideal $(y,a)$ which corresponds to the ray generated by $(1,0,1)$
in the fan. Note that $D_{0}$ has multiplicity $l$ since on the
interior of $D_{0}$ where $b,x\neq 0$, we have $a=x^{l}y^{l}b^{-1} =
\text{unit}\cdot y^{l}$.

Let $D'_{\epsilon}\subset X^{\loc}_{\con}$ be the proper transform of
$D_{\epsilon}\subset X^{\loc}_{\sing}$. For the toric case of
$\epsilon =0$, $D'_{0}$ can be determined by the combinatorics of the
fan. Indeed, as with any torus invariant divisor, $D_{0}\subset
X^{\loc}_{\sing}$ is determined by an integer valued function on the
vectors generating the rays of the fan. In this case, the function
takes the values $(0,0,0,l)$ at the generators
$((0,0,1),(0,l,1),(1,l,1),(1,0,1))$ respectively. $D_{0}$ is not
Cartier: if it were, its associated values on the vectors generating
the fan would be the restriction of a linear function on the cone. It
is however Cartier away from the singular point. This is reflected in
the fact that the values are the restrictions of linear functions on
the 2-dimensional faces of the cone. Since $F_{\con}$ is obtained from
$F_{\sing}$ by adding various new generators to the 2-dimensional
faces of $F_{\sing}$, the proper transform $D'_{0}$ is given by the
function on the generators obtained by restricting the linear function
on the faces of $F_{\sing}$ to the new generators. Namely, $D_{0}'$ is
given by the function taking values 0 on the generators $(0,j,1)$ and
$l-j$ on the generators $(1,j,1)$ for $j=0,\dotsc ,l$.

Note that $D'_{0}\subset X^{\loc}_{\con}$ is still non-Cartier: on the
cone in $F_{\con}$ generated by
$((0,i-1,1),(0,i,1),(1,i-1,1),(1,i,1))$, $D_{0}'$ take values
$(0,0,l-i+1,l-i)$. On this cone, which corresponds to the
$T$-invariant affine neighborhood of the $i$th conifold point in
$X_{\con}^{\loc}$, $D'_{0} $ is the sum of the divisors taking values
$(0,0,1,0)$ and $(0,0,l-i,l-i)$ on the generators. The latter is
principal and hence doesn't affect the blowup. The blowup of the
former gives the conifold resolution obtained by adding the
2-dimensional face spanned by $(0,i-1,1)$ and $(1,i,1)$ to the
fan. Thus
\[
X^{\loc} = \Bl_{D'_{0}}(X^{\loc}_{\con})
\]
is smooth and is given by the fan $F$ colored grey in the picture
below: \bigskip
\begin{center}
\tikz[scale=0.5]{

\draw[gray!80!black, line width=.5pt, fill=gray!80!black, line
width=.5pt, fill opacity=.1, text
opacity=1](0,0)--++(0,18)--++(6,0)--++(0,-18)--cycle; 
\draw [gray!80!black, line width=0.5pt] (0,6)--(6,6);
\draw [gray!80!black, line width=0.5pt] (0,12)--(6,12);
\draw [gray!80!black, line width=0.5pt] (0,0)--(6,6);
\draw [gray!80!black, line width=0.5pt] (0,6)--(6,12);
\draw [gray!80!black, line width=0.5pt] (0,12)--(6,18);

\draw [red, very thick] (2,-2)--++(2,4)--++(-2,2)--++(2,4)--++(-2,2)--++(2,4)--++(-2,2)--++(2,4);
\draw [red, very thick] (-2,2)--++(4,2)--++(2,4)--++(4,2);
\draw [red, very thick] (-2,8)--++(4,2)--++(2,4)--++(4,2);
\draw [red, very thick] (-2,14)--++(4,2)--++(2,4);
\draw [red, very thick] (4,2)--++(4,2);

\node[blue] at (-0.5,-0.75){$0$};
\node[blue] at (-0.5,5.25){$0$};
\node[blue] at (-0.5,11.25){$0$};
\node[blue] at (-0.5,17.75){$0$};
\node[blue] at (6.5,-0.75){$3$};
\node[blue] at (6.5,5.25){$2$};
\node[blue] at (6.5,11.25){$1$};
\node[blue] at (6.5,17.75){$0$};

\node[] at (-1.25,0.25){$(0,0,1)$};
\node[] at (-1.25,6){$(0,1,1)$};
\node[] at (-1.25,12){$(0,2,1)$};
\node[] at (-1.25,18.5){$(0,3,1)$};
\node[] at (7.5,0.25){$(1,0,1)$};
\node[] at (7.5,6){$(1,1,1)$};
\node[] at (7.5,12){$(1,2,1)$};
\node[] at (7.5,18.5){$(1,3,1)$};

\node[] at (3,3.7){$c_{1}$};
\node[] at (3,9.7){$c_{2}$};
\node[] at (3,15.7){$c_{3}$};
\node [] at (2.5,6.5) {$b_{1}$};
\node [] at (2.5,12.5) {$b_{2}$};

}

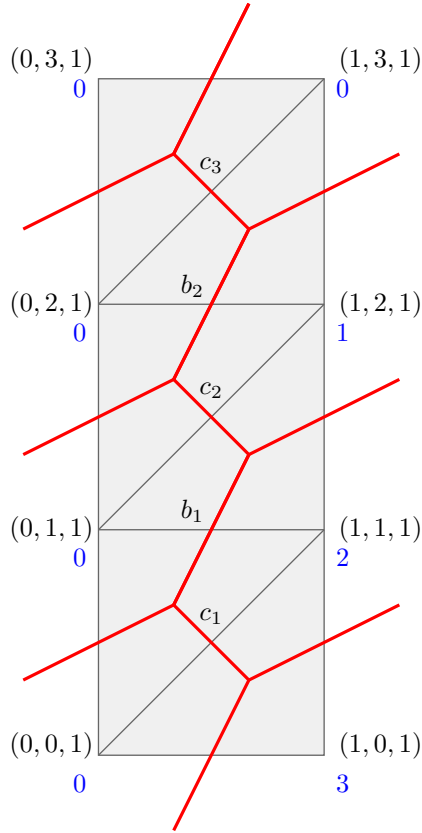
The fan $F$ of $X^{\loc}$, the dual
polytope, and the multiplicities of 
$D''_{0}$.
\end{center}

\bigskip

Let $D''_{0}\subset X^{\loc}$ be the proper transform of $D_{0}'$. The
values of $D_{0}''$ on the generators of the fan $F$ are depicted in
blue in the above picture and the dual polytope is depicted in red.

The exceptional locus of $X^{\loc}\to X_{\sing}^{\loc}$ is given by
$c_{1}\cup b_{1}\cup \dotsb \cup b_{l-1}\cup c_{l}$, a chain of $2l-1$
$\PP^{1}$s. The $l$ curves $c_{1},\dotsc ,c_{l}$ are the expectional
curves of the conifold resolution $X^{\loc}\to X^{\loc}_{\con}$ and
the $l-1$ curves $b_{1},\dotsc ,b_{l-1}$ are the proper transforms of
the exceptional curves of $X_{\con}^{\loc}\to X^{\loc}_{\sing}$.

We can now compute the intersection numbers $D_{0}''\cdot c_{i}$
and $D_{0}''\cdot b_{i}$. Let $D[r,s,1]$ denote the torus invariant
divisor corresponding to the ray generated by $(r,s,1)$. So then 
\[
D''_{0} = \sum_{k=0}^{l} (l-k) D[1,k,1].
\]
Now the intersection of the proper curves $b_{i}$ and $c_{i}$ with
$D[r,s,1]$ can be computed using standard toric geometry. The
intersection number is 0 if the curve and divisor are disjoint and 1
if they meet in a single point. If the curve is contained in the
divisor, one can find a linearly equivalent divisor (by adding a
suitable global linear function on the fan) such that no component of
the new divisor contains the curve and then easily compute the
intersection number. The results are
\begin{align*}
b_{i}\cdot D[1,k,1] &= \begin{cases}
1&\text{if $k=i+1$}\\
-1&\text{if $k=i$}\\
0&\text{otherwise.}\\
\end{cases}\\
c_{i}\cdot D[1,k,1] &= \begin{cases}
1&\text{if $k=i-1$}\\
-1&\text{if $k=i$}\\
0&\text{otherwise.}\\
\end{cases}
\end{align*}
We then get
\begin{align*}
b_{i}\cdot D''_{0} & = l-(i+1) - (l-i)\\
&=-1\\
c_{i}\cdot D''_{0} & = l-(i-1) - (l-i)\\
&=1.
\end{align*}

We summarize the above discussion in the following lemma.
\begin{lemma}\label{lem: Xloc is smooth and intersection numbers}
Let $X^{\loc}=\Bl_{D_{0}'}(X^{\loc}_{\con})$. Then $X^{\loc}$ is a
smooth toric CY3. The exceptional set of the resolution $X^{\loc}\to
X^{\loc}_{\sing}$ is a chain of $2l-1$ curves $c_{1}\cup b_{1}\cup
\dotsb \cup b_{l-1}\cup c_{l}$ where the $c_{i}$'s are the exceptional
curves of $X^{\loc}\to X^{\loc}_{\con}$ and the $b_{i}$'s are the
proper transforms of the exceptional curves of $X^{\loc}_{\con}\to
X^{\loc}_{\sing}$. Let $D_{0}''\subset X^{\loc}$ be the proper
transform of $D_{0}'.$ Then the intersection numbers of $D_{0}''$ with
the exceptional curves are given by
\begin{align}
D_{0}''\cdot c_{i} &= +1\text{, for } i=1,\dotsc ,l,\\
D_{0}''\cdot b_{i}&=-1\text{, for } i=1,\dotsc ,l-1.\label{eqn:D0.bi=-1}
\end{align}
\end{lemma}

We next show that the above lemma holds for the whole family of
divisors parameterized by $\epsilon$. This will allow us to convert
the toric result about $X^{\loc}$ into a global result for $X$ since
the divisor $\Gamma =D'_{\epsilon=1}$ is well defined globally.

Let 
\[
\DD \subset \AA^{1}\times X^{\loc}_{\sing}
\]
be the Weil divisor with ideal
\[
(y^{l}-\epsilon b,a-\epsilon x^{l})\subset \CC [x,y,a,b,\epsilon ]/(ab-x^{l}y^{l}).
\]
Let $\DD '$ be the proper transform of $\DD$ under $\AA^{1}\times
X^{\loc}_{\con}\to \AA^{1}\times X^{\loc}_{\sing}$ and consider
\[
\XX^{\loc} = \Bl_{\DD '}(\AA^{1}\times X^{\loc}_{\con}). 
\]

By the functoriality of blowups, the fiber of $\XX^{\loc}\to \AA^{1}$
over $\epsilon$ is given by
$\Bl_{D'_{\epsilon}}(X^{\loc}_{\con})$. Then since we've shown that
$\Bl_{D'_{0}}(X^{\loc}_{\con})$ is non-singular, it follows that
$\Bl_{D'_{\epsilon}}(X^{\loc}_{\con})$ is non-singular for generic
$\epsilon$ and hence for all $\epsilon$. In fact, since the
singularities are all conifolds and the resolutions are conifold
resolutions, we have $\XX^{\loc} = X^{\loc}\times \AA^{1}$.  Let $\DD
''\subset \XX^{\loc}$ be the proper transform of $\DD '$. Then
$D_{\epsilon}''\cdot c_{i}$ and $D''_{\epsilon }\cdot b_{i}$ are
independent of $\epsilon$ since they are given by $\deg (\OO (\DD
'')|_{\epsilon \times c_{i}})$ and $\deg (\OO (\DD '')|_{\epsilon
\times b_{i}})$.

In particular, when $\epsilon =1$, the divisor $D_{\epsilon
=1}'\subset X^{\loc}_{\con}$ is a local model for $\Gamma \subset
X_{\con}$ and then the local results for $D'_{\epsilon =1}$ imply the
following global results:

\begin{proposition}\label{prop: Blowup of Gamma is smooth,
intersection numbers of Gamma with exceptional curves} Let $\Gamma
\subset X_{\con}$ be as defined in \S~\ref{sec: the construction of
XN}. Let $X_{N}=\Bl_{\Gamma}(X_{\con})$ and let us also denote by
$\Gamma\subset X_{N}$ the proper transform of $\Gamma\subset
X_{\con}$. Then $X_{N}$ is a smooth CY3 and
\begin{equation}\label{eqn: Gamma.c=1}
\Gamma \cdot c = 1
\end{equation}
for any exceptional curve $c$ of $X\to X_{\con}$ and
\begin{equation}\label{eqn: Gamma.b=-1}
\Gamma \cdot b = -1
\end{equation}
for any curve $b$ given by the proper transform of an exceptional
curve of $X_{\con}\to X_{\sing}$. 
\end{proposition}

\section{Divisors, curves, and intersection numbers on
$X_{N}$}\label{sec: Divs curves and int numbers on XN}

We now study the cohomology classes of $X_{N}$ and $\XtildeN$. In
particular we find a $\QQ$ basis for curve classes and divisor classes
and we compute their intersections numbers.

The divisor classes of $X_{N}$ necessarily consist of the divisor
classes of the generic fiber and the fiber classes. The generic fiber
is spanned by the three divisors
\[
S = S\times_{\PP^{1}}\{0 \},\quad S' = \{0 \}\times_{\PP^{1}}S',\quad
\text{and $\quad \Gamma$.}
\]
The fiber divisor classes are given by the generic fiber $F$ and the
irreducible components of the singular fibers.

We now develop notation for the curves and divisors in the singular
fibers of $X_{N}\to \PP^{1}$. As discussed in Section~\ref{sec: the
construction of XN}, for each $k\in \Theta_{N}$ there is a
corresponding singular fiber of $X_{N}\to \PP^{1}$ which is the proper
transform of $I_{k}\times I_{l}\subset X_{\con}$ where we recall that 
\[
l=\frac{N}{k}.
\]
Fibers of this type are called multi-banana fibers and were studied by
Kanazawa-Lau \cite{Kanazawa-Lau} and Morishige
\cite{Morishige-multi-banana}. We label\footnote{If $k$ is repeated in
$\Theta_{N}$ we will use primes to further distinguish so that the
labels will be in $(1,2,3,6)$, $(1,1',5,5')$, $(2,2',4,4')$, or
$(3,3',3'',3''')$.} such a fiber by $F(k)$.

The singular fiber $F(k)\subset X_{N}$ is a non-normal toric surface
and its formal neighborhood $\Fhat (k)$ has a universal cover
which is a formal toric Calabi-Yau threefold $\widehat{U}$ modelled on
the toric CY3 whose fan in $\RR^{3}$ consists of the cones generated
by
\[
(i,j,1),\, (i+1,j,1),\, (i+1,j+1,1),\quad (i,j)\in \ZZ \times \ZZ ,
\]
and 
\[
(i,j,1),\, (i,j+1,1),\, (i+1,j+1,1),\quad (i,j)\in \ZZ \times \ZZ .
\]
An element $(s,t)\in \ZZ \times \ZZ$ in the group of deck
transformations of $\widehat{U}\to \Fhat (k)$ acts on the generators
of the cones by translation by $(sk,tl,0)$ (see
\cite[\S~2.2]{Morishige-multi-banana} for more details). We then may
choose a fundamental domain for the $\ZZ \times \ZZ$ action on the
cones, namely the cones given above with $i\in \{0,\dotsc ,k-1 \}$ and
$j\in \{0,\dotsc ,l-1 \}$. The quotient of $\Fhat (k)$ by the action
of $G_{N}$ is then given by the quotient of $\widehat{U}$ by $\ZZ
\times \ZZ$ where now $(s,t)$ acts by translation by $(s,t,0)$ on the
generators of the cones. This quotient is the formal neighborhood of a
banana fiber, $\Fhat_{\ban}$ in the notation of \cite{Bryan-Banana}:
\[
\Fhat (k)/G_{N}\cong \Fhat_{\ban}.
\]

We label the torus invariant divisors in $F(k)$ by $D_{ij}(k)$,
$i=1,\dotsc ,k,\, j=1,\dotsc ,l$ and the torus invariant curves by
$a_{ij}(k)$, $b_{ij}(k)$, and $c_{ij}(k)$. We do this such that the
zero section meets $D_{kl}(k)$ and the $b$ and $c$ curves coincide
with the $b$ and $c$ curves in the local model. In particular,
$c_{ij}(k)$ are the exceptional curves of the conifold resolution
$X_{N}\to X_{\con }$, $b_{ij}(k)$ for $ j\neq l$ are the proper
transforms of the exceptional curves of $X_{\con}\to X_{\sing}$, and
$b_{il}(k)$ is the proper transform of the curve
\[
n_{i}\times I_{1}^{\ZZ_{l}}\subset I_{k}\times
I_{1}^{\ZZ_{l}}\subset X_{\sing}.
\]
Finally, $a_{ij}(k)$ is the proper transform of 
\[
a_{i}\times n_{j}\subset I_{i}\times I_{l}\subset X_{\con}
\]
where $a_{i}$ is the $i$th curve in $I_{i}$ and $n_{j}$ is the $j$th
node in $I_{l}$.

We illustrate this below for $N=6$, $k=3$, and $l=2$ (we suppress the $k$
from the notation in the diagram):

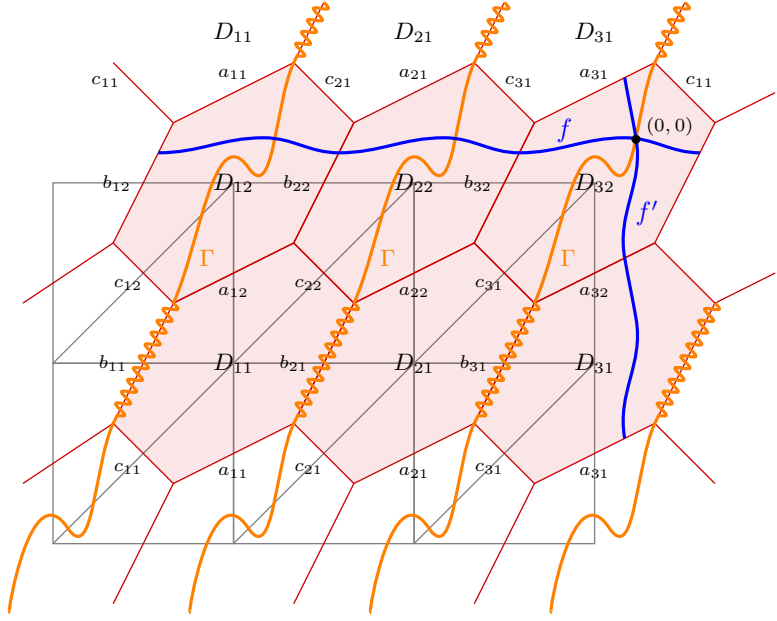
\begin{center}

\tikz[scale=1]{
\foreach \x\y\z in {0/0/A,2.4/0/B,4.8/0/C,0/-2.4/D,2.4/-2.4/E,4.8/-2.4/F}
\node[draw, line width=.5pt, black!50, minimum size=2.4cm, outer sep=0pt] (\z) at (\x,\y) {};
\foreach \i in {A,B,...,F}
\draw[black!50, line width=.5pt] (\i.north east)--(\i.south west);

\foreach \i[count=\j] in {A,B,...,F}
\draw[red!80!black, line width=.5pt, fill=red!80!black, line
width=.5pt, fill opacity=.1, text opacity=1]
(\i.center)--($(\i.center)+(-4mm,4mm)$)--($(\i.north)+(4mm,8mm)$)--
coordinate[pos=1](topvertex\j)
++(16mm,8mm)--++(8mm,-8mm)--++(-8mm,-16mm)--
coordinate[pos=1](bottomvertex\j) ($(\i.center)+(4mm,-4mm)$)--cycle; 

\coordinate(n1) at ($(A.north)+(4mm,8mm)$);
\coordinate(n2) at ($(C.north)+(28mm,8mm)$);
\coordinate(n2a) at ($(n2)+(16mm,8mm)$);
\coordinate(n3) at ($(F.north)+(28mm,8mm)$);
\coordinate(n4) at ($(F.center)+(20mm,4mm)$);
\coordinate(n5) at ($(A.center)+(-4mm,4mm)$);
\coordinate(n6) at ($(D.center)+(-4mm,4mm)$);
\foreach \i in {1,2,3}
\draw[red!80!black, line width=.5pt] (topvertex\i)--+(4mm,8mm);

\foreach \i in {4,5,6}
\draw[red!80!black, line width=.5pt] (bottomvertex\i)--+(-8mm,-16mm);
\draw[red!80!black, line width=.5pt] (n1)--+(-8mm,8mm);

\draw[red!80!black, line width=.5pt] (n2)--+(8mm,4mm);
\draw[red!80!black, line width=.5pt] (n3)--+(8mm,4mm);
\draw[red!80!black, line width=.5pt] (n4)--+(8mm,-8mm);

\draw[red!80!black, line width=.5pt] (n5)--+(-12mm,-8mm);
\draw[red!80!black, line width=.5pt] (n6)--+(-12mm,-8mm);

\foreach \i in {1,2,3}{
\draw[very thick, decorate,decoration={coil,aspect=0, segment
length=5pt}, orange] ($(topvertex\i)+(4mm,8mm)$) -- (topvertex\i); 
\draw[very thick, orange] (topvertex\i) to[out=245, in=-55]
++(-6mm,-14mm) to[out=125, in=65] (bottomvertex\i);

\draw[very thick, blue] (topvertex\i)+(-1.8cm,-1.2cm) to[out=0, in=180]
++(-4mm,-10mm) to[out=0, in=180] ++(1cm,-0.2cm);
}

\draw[very thick, blue] (64mm,26mm) to[out=-80, in=100]
++(1.5mm,-8mm) to[out=-80, in=100] ++(-1.5mm,-16mm);
\draw[very thick, blue] (64mm,2mm) to[out=-80, in=100]
++(1.5mm,-8mm) to[out=-80, in=100] ++(-1.5mm,-16mm);
\node [blue] at(67mm,8mm) {$f'$};
\node [blue] at(56mm,19mm) {$f$};
\node [] at(70mm,19.5mm) {\scriptsize $(0,0)$};
\fill [black] (65.5mm,17.8mm) circle(0.6mm);

\foreach \i\j in {1/D,2/E,3/F}{
\draw[very thick, decorate,decoration={coil,aspect=0, segment length=5pt}, orange] (bottomvertex\i) -- ($(\j.center)+(-4mm,4mm)$);
\draw[very thick, orange, shorten >=10mm] ($(\j.center)+(-4mm,4mm)$)
to[out=245, in=-55] ++(-6mm,-14mm) to[out=123, in=65] ++(-12mm,-20mm); 
}

%
\draw[very thick, decorate,decoration={coil,aspect=0, segment length=5pt}, orange] (n3) -- (n4);
\draw[very thick, orange, shorten >=10mm] (n4) to[out=245, in=-55] ++(-6mm,-14mm) to[out=123, in=65] ++(-12mm,-20mm);

\foreach \i\j in {A/$D_{11}$,B/$D_{21}$,C/$D_{31}$}
\node[] at ([yshift=2cm]\i.north east)  {\small \j};

\foreach \i\j in {A/$D_{12}$,B/$D_{22}$,C/$D_{32}$}
\node[] at ([yshift=0cm]\i.north east)  {\small  \j};

\foreach \i\j in {A/$D_{11}$,B/$D_{21}$,C/$D_{31}$}
\node[] at ([yshift=0cm]\i.south east)  {\small  \j};
\foreach \i\j in {A/$b_{12}$,B/$b_{22}$,C/$b_{32}$}
\node[] at ([xshift=-3.5mm ,yshift=0cm]\i.north)  {\scriptsize \j};

\foreach \i\j in {A/$\Gamma $,B/$\Gamma $,C/$\Gamma $}
\node[orange] at ([xshift=-3.5mm ,yshift=-1.0cm]\i.north east)  {\small \j};

\foreach \i\j in {A/$a_{11}$,B/$a_{21}$,C/$a_{31}$}
\node[] at ([xshift=0mm ,yshift=1.45cm]\i.north east)  {\scriptsize \j};

\foreach \i\j in {A/$c_{21}$,B/$c_{31}$,C/$c_{11}$}
\node[] at ([xshift=1.4cm ,yshift=1.35cm]\i.north east)  {\scriptsize \j};

\node[] at ([xshift=-5mm ,yshift=1.35cm]A.north)  {\scriptsize $c_{11}$};

\foreach \i\j in {A/$a_{12}$,B/$a_{22}$,C/$a_{32}$}
\node[] at ([xshift=0mm ,yshift=-2.5mm]\i.east)  {\scriptsize \j};

\foreach \i\j in {A/$c_{12}$,B/$c_{22}$,C/$c_{31}$}
\node[] at ([xshift=-2mm ,yshift=-1.5mm]\i.center)  {\scriptsize \j};

\foreach \i\j in {D/$b_{11}$,E/$b_{21}$,F/$b_{31}$}
\node[] at ([xshift=-4mm ,yshift=0cm]\i.north)  {\scriptsize \j};

\foreach \i\j in {D/$c_{11}$,E/$c_{21}$,F/$c_{31}$}
\node[] at ([xshift=-2mm ,yshift=-2mm]\i.center)  {\scriptsize \j};

\foreach \i\j in {D/$a_{11}$,E/$a_{21}$,F/$a_{31}$}
\node[] at ([xshift=0mm ,yshift=-2.5mm]\i.east)  {\scriptsize \j};
}

 A fundamental domain for the toric diagram
(gray) and dual polytope (red) for $\Fhat (3)$, the formal
neighborhood of the singular fiber $F(3)\subset X_{6}$.
\end{center}
\bigskip

Note that in the above example, the fan and dual polytope of the local
model appears $k=3$ times, each with $l=2$. The intersection of
$\Gamma$ with $F(3)$ is depicted in orange.

The $4N+4$ divisor classes given by 
\[
\{F,S,S',\Gamma ,D_{ij}(k) \}_{k\in \Theta_{N},\, i=1,\dotsc ,k,\, j=1,\dotsc ,l}
\]
admit four obvious relations, namely
\[
\sum_{i,j}D_{ij}(k) = F
\]
for each $k\in \Theta_{N}$. Since the Picard number of $X_{N}$ is 4N,
there are no further relations.

We obtain the following intersection numbers easily since in each
case, the curve and divisor are either disjoint or intersect
transversely in a single point:
\[
S\cdot a_{ij}(k) = S'\cdot b_{ij}(k)=S\cdot c_{ij}(k) = S'\cdot
c_{ij}(k) = 0.
\]
and
\begin{align*}
S'\cdot a_{ij}(k) &= \begin{cases} 1&\text{if $i=k$,}\\ 0& \text{if $i\neq k$,}\end{cases}\\
S\cdot b_{ij}(k) &= \begin{cases} 1&\text{if $j=l$,}\\ 0 &\text{if $j\neq l$.}\end{cases}\\
\end{align*}
By Equations~\eqref{eqn: Gamma.c=1} and \eqref{eqn: Gamma.b=-1} in
Proposition~\ref{prop: Blowup of Gamma is smooth, intersection numbers
of Gamma with exceptional curves}, we have
\begin{align*}
\Gamma \cdot c_{ij}(k) &=1,\\
\Gamma \cdot b_{ij}(k) &=-1 \quad  \text{ for $j\neq l$.}\\
\end{align*}
Thus to get all the intersection numbers of the fiber curves with
$\{S,S',\Gamma \}$, it remains to compute $\Gamma \cdot a_{ij}(k)$ and
$\Gamma \cdot b_{il}(k)$. To determine these, we use the curve classes
\[
f=S\cdot F,\quad f' = S'\cdot F
\]
which are the fiber curve classes of the elliptic fibrations $S\to
\PP^{1}$ and $S'\to \PP^{1}$. We may write these classes in terms of
$a_{ij}(k),b_{ij}(k),c_{ij}(k)$ since they are given as the total
transform of the curves $I_{k}\times \{0 \}$ and $\{0 \}\times I_{l}$
respectively in the fiber $F(k)\subset X_{N}$. These curves are
illustrated in the above diagram in blue.

Since $I_{k}\times \{0 \}$ is homologous to $I_{k}\times n_{j}$ where
$n_{j}$ is any node, and the total transform of $I_{k}\times n_{j}$ is
the union 
\[
\bigcup_{i}\,\, c_{ij}(k) \cup a_{ij}(k)
\]
we get the following relation for any fixed $k$
and $j$:
\[
f = \sum_{i=1}^{k}  c_{ij}(k)+a_{ij}(k).
\]
Similarly, for any $k$ and $i$ we get
\[
f' = \sum_{j=1}^{l} c_{ij}(k)+b_{ij}(k).
\]

Taking a representative for $f'$ and $f$ in the generic fiber, we see
that\footnote{We remark that the picture of the intersection of $f$
and $\Gamma$ in the diagram is misleading. The curve $f$ should meet
$\Gamma$ twice in each of $D_{12},D_{22},D_{32}$ for a total of
$N=6$. }
\[
f'\cdot \Gamma =1,\quad \quad f\cdot \Gamma =N.
\]
Combining, we get
\begin{align*}
1&=f'\cdot \Gamma \\
&=\sum_{j=1}^{l}\left(c_{ij}(k)+b_{ij}(k) \right)\cdot \Gamma \\
&=l-(l-1)+b_{il}(k)\cdot \Gamma 
\end{align*}
which implies
\[
b_{il}(k)\cdot \Gamma =0
\]
for any $i$ and $k$.

We note that $a_{ij}(k)\cdot \Gamma$ is independent of $i$ since the
action of the first factor of $G_{N}\times G_{N}$ preserves $\Gamma$
and permutes the $i$ indices of $a_{ij}(k)$. Then for any $j$ and $k$
we have
\begin{align*}
N&=f\cdot \Gamma \\
&=\sum_{i=1}^{k}\left(c_{ij}(k) +a_{ij}(k) \right)\cdot \Gamma \\
&=k+ka_{ij}(k)\cdot \Gamma 
\end{align*}
where the last equality holds for all $i,j,k$. Thus we get
\[
a_{ij}(k)\cdot \Gamma = l-1.
\]

We summarize these results in the following
\begin{lemma}\label{lem: intersection matrix of aij,bij,cij with S,S',Gamma}
The intersection numbers of the fiber curve classes
$a_{ij}(k),b_{ij}(k),c_{ij}(k)$ with the divisors $S,S',\Gamma$ are
given by the following table:
\begin{center}
\begin{tabular}{cccc}
$\epsilon $  &	$S'\cdot \epsilon $&	$S\cdot \epsilon $&	$\Gamma\cdot \epsilon  $\\ \hline
$a_{ij}(k)$&   $\left\{\begin{smallmatrix} 1&\text{if
$i=k$}\\0&\text{if $i\neq k$} \end{smallmatrix}
\right\}$& $0$&		 $l-1$\\
$b_{ij}(k)$&$0$ &	   $\left\{\begin{smallmatrix} 1&\text{if
$j=l$}\\0&\text{if $j\neq l$} \end{smallmatrix}
\right\}$&  $\left\{\begin{smallmatrix} 0&\text{if
$j=l$}\\-1 &\text{if $j\neq l$} \end{smallmatrix}
\right\}$\\
$c_{ij}(k)$&	$0$ &	$0$ &	$1$
\end{tabular}
\end{center}
\end{lemma}

We are primarily interested in curve and divisor classes on $\XtildeN
=X_{N}/G_{N}$. Let
\[
\pi :X_{N}\to \XtildeN 
\]
be the quotient map. We note that for each $k$, $G_{N}$ acts simply
transitively on the sets $\{a_{ij}(k) \},\{b_{ij}(k) \},\{c_{ij}(k)
\}$, and $\{D_{ij}(k) \}$. In particular the classes
\[
\tilde{a}(k) = \pi_{*}(a_{ij}(k)),\quad  \tilde{b}(k) = \pi_{*}(b_{ij}(k)),\quad  \tilde{c}(k) = \pi_{*}(c_{ij}(k))
\]
are independent of the choice of $i$ and $j$. We also define
\[
\Gammatilde =\pi_{*}(\Gamma ),\quad  \Stilde =\pi_{*}(S),\quad  \Sprimetilde =\pi_{*}(S')
\]
and finally we define $\Ftilde$ to be the class of the fiber of
$\XtildeN \to \PP^{1}$ so that (in slight notational conflict with the
other tilded classes)
\[
\Ftilde =\tfrac{1}{N}\pi_{*}(F).
\]

The divisor classes $\{\Ftilde ,\Stilde ,\Sprimetilde ,\Gammatilde \}$
span $\Pic (\XtildeN )$ and to compute their intersection numbers with
the fiber curve classes $\atilde(k)$, $\btilde(k)$, $\ctilde(k)$ we
use the following.

If $D$ is any divisor class on $X_{N}$, $\epsilontilde$ is a curve
class on $\XtildeN $, and $\widetilde{D}=\pi_{*}(D)$ then
\begin{align*}
\widetilde{D}\cdot \epsilontilde  &= \pi_{*}(D)\cdot \epsilontilde \\
&= \pi_{*}(D\cdot \pi^{*}\epsilontilde )\\
&= D\cdot \pi^{*}\epsilontilde 
\end{align*}
so to compute $\Stilde \cdot \atilde (k)$ for example, we need to
compute
\[
S\cdot \pi^{*}(\atilde (k)) = S\cdot \sum_{i,j}a_{ij}(k)
\]
and this can be read off from the table in Lemma~\ref{lem:
intersection matrix of aij,bij,cij with S,S',Gamma}. Carrying this out
we get the following table of intersections on $\XtildeN$.
\begin{center}
\begin{tabular}{cccc}
$\epsilontilde$&	$\Sprimetilde \cdot \epsilontilde$&	$\Stilde \cdot \epsilontilde$&	$\Gammatilde \cdot \epsilontilde$\\
\hline
$\atilde (k)$&	$l$& $0$&	$N(l-1)$\\	
$\btilde (k)$&	$0$& $k$&	$-k(l-1)$\\	
$\ctilde (k)$&	$0$& $0$&	$N$\\	
\end{tabular}
\end{center}

Geometrically, the classes $\atilde (k),\btilde (k),\ctilde (k)$ are
represented by the three banana curves in the fiber $\Ftilde (k) =
F(k)/G_{N}$.  It will be convenient to consider a slightly different
basis of curve and divisor classes. Namely, consider the curve classes
\[
\atilde (k)+\ctilde (k),\quad \btilde (k)+\ctilde (k), \quad \ctilde (k)
\]
and the divisor class
\[
\Deltatilde =\Gammatilde -N\Sprimetilde -\Stilde .
\]

\begin{lemma}\label{lem: curve divisor pairings in the diagonal basis}
The intersection numbers for the above classes are given in the following:
\begin{center}
\begin{tabular}{cccc}
$\epsilontilde$&	$\Sprimetilde \cdot \epsilontilde$&	$\Stilde \cdot \epsilontilde$&	$\Deltatilde \cdot \epsilontilde$\\
\hline
$\atilde (k)+\ctilde (k)$&	$l$& $0$&	$0$\\	
$\btilde (k)+\ctilde (k)$&	$0$& $k$&	$0$\\	
$\ctilde (k)$&	$0$& $0$&	$N$\\
\end{tabular}
\end{center}
\end{lemma}

\bigskip

\subsection{The intersection form of a smooth fiber.}

The divisor classes 
\[
\Gamma ,\quad S,\quad S',\quad \Delta =\Gamma -NS'-S
\]
on $X_{N}$ and the corresponding classes
\[
\Gammatilde ,\quad \Stilde ,\quad \Sprimetilde ,\quad \Deltatilde 
\]
on $\XtildeN$ restrict to a fiber to give curve classes
\[
\gamma =\Gamma \cdot F,\quad f=S\cdot F,\quad f'=S'\cdot F,\quad \delta =\Delta \cdot F
\]
on $X_{N}$, and correspondingly
\[
\gammatilde =\Gammatilde  \cdot \Ftilde ,\quad \ftilde =\Stilde \cdot \Ftilde,\quad
\fprimetilde =\Sprimetilde \cdot \Ftilde,\quad \deltatilde  =\Deltatilde  \cdot \Ftilde
\]
on $\XtildeN$.

Viewed as divisor classes on a smooth fiber $E\times E'\subset X_{N}$
or $(E\times E')/G_{N}\subset \XtildeN$, these classes are endowed
with an intersection form.  Geometrically, we may explicitly write
these cycles:
\[
f=\{(x,0)\in E\times E' \},\quad f'=\{(0,x')\in E\times E' \},\quad
\gamma =\{(x,g(x))\in E\times E' \} 
\]
where $g:E\to E'$ is the quotient map.

The self-intersection of $f,f',\gamma$ are all zero since translation
by a generic element in $E\times E'$ creates a disjoint homologous
cycle. The remaining intersections are transverse and easily counted
using the above descriptions. They are given by
\[
f\cdot f'=1,\quad f\cdot \gamma =N,\quad f'\cdot \gamma =1.
\]

It follows then for $\delta =\gamma -Nf'-f$ that
\[
\delta \cdot f=0,\quad  \delta \cdot f'=0,\quad \delta \cdot \delta =-2N
\]
so in the basis $\left\langle \delta ,f,f' \right\rangle$ for $\Pic
(E\times E')$ the intersection form is
\[
\begin{pmatrix}
-2N&0&0\\
0&0&1\\
0&1&0
  \end{pmatrix}.
\]

All classes $\alpha \in \{f,f',\gamma \}$ have the property that 
\begin{align*}
\pi^{*}\pi_{*}\alpha & = \text{$G_{N}$ orbit of $\alpha$}\\
& = N\alpha 
\end{align*}
so for $\widetilde{\alpha},\widetilde{\beta }\in \{\ftilde
,\fprimetilde ,\gammatilde ,\deltatilde \}$ we have 
\[
\widetilde{\alpha}\cdot \widetilde{\beta} = N\alpha \cdot \beta 
\]
since
\begin{align*}
\alpha \cdot \beta &= \pi_{*}(\alpha \cdot \beta )\\
&=\tfrac{1}{N}\pi_{*}(\pi^{*}\pi_{*}\alpha \cdot \beta )\\
&=\tfrac{1}{N}(\pi_{*}\alpha \cdot\pi_{*} \beta )\\
&=\tfrac{1}{N}\widetilde{\alpha}\cdot \widetilde{\beta}.
\end{align*}
Thus we've proved the following
\begin{lemma}\label{lem: intersection form on smooth fiber of XtildeN}
The classes $\deltatilde ,\ftilde ,\fprimetilde$ given by the
restriction of the divisor classes $\Deltatilde ,\Stilde
,\Sprimetilde$ in $\XtildeN$ to a smooth fiber have intersection
pairing given by the matrix
\[
\begin{pmatrix}
-2N^{2} &	0 &	0 \\
0       & 	0 &	N \\
0 	& 	N &	0    	
\end{pmatrix}.
\]
\end{lemma}

It is also useful to write the classes of the banana curves $\atilde
(k),\btilde (k),\ctilde (k)\subset \Ftilde (k)$ in terms of the curve
classes $\deltatilde ,\ftilde ,\fprimetilde$.
\begin{lemma}\label{lem: atilde,btilde,ctilde in terms of
ftilde,fprimetilde, deltatilde}
The following equations of curve classes hold
\begin{align*}
\atilde (k)+\ctilde (k) &= \frac{1}{k}\ftilde \\
\btilde (k)+\ctilde (k) &= \frac{1}{l}\fprimetilde \\
\ctilde (k) &= -\frac{1}{2N}\deltatilde .
\end{align*}
\end{lemma}
\begin{proof}
It suffices to show that both sides of each equation have the same
pairing with the divisors $\Sprimetilde ,\Stilde ,$ and $\Deltatilde
$. The pairings of $\ftilde ,\fprimetilde ,\deltatilde$ with $\Stilde
,\Sprimetilde ,\Deltatilde$ are given by the intersection pairing in
Lemma~\ref{lem: intersection form on smooth fiber of XtildeN}. In
particular, the only non-zero pairings are
\begin{align*}
\frac{1}{k} \ftilde \cdot \Sprimetilde &= \frac{N}{k}=l\\
\frac{1}{l} \fprimetilde \cdot \Stilde &= \frac{N}{l}=k\\
-\frac{1}{2N} \deltatilde \cdot \Deltatilde &= -\frac{1}{2N}(-2N^{2})
= N. 
\end{align*}
The pairing with the classes on the left are given by Lemma~\ref{lem:
curve divisor pairings in the diagonal basis} which are the same.
\end{proof}

\section{Computation of the enumerative invariants of
$\XtildeN$}\label{sec: computation of DT/GW/GV of XN}

\subsection{Review of the partition function of Banana manifolds} The
fiber curve DT partition function of the generic banana manifold
$X_{\ban}$ was computed in \cite{Bryan-Banana} with a combination of
motivic and topological vertex methods. The method works for any other
CY3 of banana type including $\XtildeN$. What is needed is a CY3 $X$
with an Abelian surface fibration $X\to \PP^{1}$ whose singular fibers
$F_{1},\dotsc ,F_{r}$ are banana fibers, so that in particular the
formal neighborhoods $\Fhat_{i}$ are isomorphic to $\Fhat_{\ban}$ see
\cite[\S~4]{Bryan-Banana}. The result is that the fiber curve DT
partition function is a product over the contributions from singular
fibers
\[
Z_{X} = \prod_{i=1}^{r} Z_{\Fhat_{i}}
\]
where each $Z_{\Fhat_{i}}$ is, up to a change of variables, given by
\begin{align*}
Z_{\Fhat_{\ban}}(p,Q_{1},Q_{2},Q_{3}) &= \sum_{\dvec,m}
\DT_{m,\beta_{\dvec}} (-p)^{m}Q_{1}^{d_{1}}Q_{2}^{d_{2}}Q_{3}^{d_{3}} \\
&= \prod_{\dvec ,m} (1-p^{m}Q_{1}^{d_{1}}Q_{2}^{d_{2}}Q_{3}^{d_{3}})^{-c(\dvec^{2},m)}
\end{align*}
where if $a,b,c$ are the banana curves in $\Fhat_{\ban}$,
\[
\beta_{\dvec} = d_{1}a+d_{2}b+d_{3}c.
\]
Moreover, the coefficients $c(\dvec^{2},m)$ are defined as in Equation~\eqref{eqn:
definition of the coefficientsc(a,m)} and $\dvec^{2}$ is by
definition
\[
\dvec^{2} = 2d_{1}d_{2}+ 2d_{2}d_{3}+ 2d_{3}d_{1} - d_{1}^{2} - d_{2}^{2} - d_{3}^{2}.
\]
Finally, in the above product we require $d_{1},d_{2},d_{3}\geq 0$
and if $d_{1}=d_{2}=d_{3}=0$ we require $m>0$.

\subsection{The DT partition function of $\XtildeN$}
In the case of $\XtildeN$, we have four singular fibers, indexed by
$k\in \Theta_{N}$ given by 
\[
\Fhat (k)/G_{N} \cong \Fhat_{\ban}
\]
and containing banana curves $\atilde (k),\btilde (k),\ctilde
(k)$. Therefore
\[
Z_{\XtildeN} = \prod_{k\in \Theta_{N}} Z_{\Fhat (k)/G_{N}}
\]
and $ Z_{\Fhat (k)/G_{N}}=Z_{\Fhat_{\ban}}$ up to a change of
variables which we determine as follows.

The contribution of a class
\[
\betatildesubd =d_{1}\atilde (k)+d_{2}\btilde (k)+d_{3}\ctilde (k)
\]
to the DT partition function contributes by definition to the
coefficient of the monomial
\[
y^{\betatildesubd \cdot \Deltatilde} \, q^{\betatildesubd \cdot
\Sprimetilde}\,Q^{\betatildesubd \cdot \Stilde}. 
\]
By writing
\[
\betatildesubd  = d_{1}(\atilde (k)+\ctilde (k))+ d_{2}(\btilde
(k)+\ctilde (k)) +(d_{3}-d_{1}-d_{2})\ctilde (k) 
\]
we can easily rewrite the above monomial using the table in
Lemma~\ref{lem: curve divisor pairings in the diagonal basis}. It is
given by
\begin{align*}
y^{\betatildesubd \cdot \Deltatilde} \cdot  q^{\betatildesubd \cdot
\Sprimetilde} \cdot Q^{\betatildesubd \cdot \Stilde} &=
y^{N(d_{3}-d_{1}-d_{2})} \cdot q^{ld_{1}}\cdot Q^{kd_{2}} \\
&= \left(q^{l}y^{-N} \right)^{d_{1}} \left(Q^{k}y^{-N} \right)^{d_{2}}
\left(y^{N} \right)^{d_{3}}
\end{align*}
where recall that $l=N/k$. So
\[
Z_{\Fhat (k)/G_{N}}(p,y,q,Q) = Z_{\Fhat_{\ban}}(p,Q_{1},Q_{2},Q_{3})
\]
where
\begin{equation}\label{eqn: Qi to q,Q,y change of vars}
Q_{1} = q^{l}y^{-N},\quad Q_{2} = Q^{k}y^{-N},\quad Q_{3} = y^{N}.
\end{equation}

Combining, we've shown that
\[
Z_{\XtildeN}(p,y,q,Q) = \prod_{k\in \Theta_{N}} \prod_{\dvec ,m}
\left(1-p^{m}y^{N(d_{3}-d_{1}-d_{2})}q^{ld_{1}}Q^{kd_{2}}
\right)^{-c(\dvec^{2},m)}. 
\]
Letting
\[
r=d_{1},\quad s=d_{2},\quad t=d_{3}-d_{1}-d_{2}
\]
and observing that 
\[
\dvec^{2} = 4 rs-t^{2}
\]
we then have the formula for $Z_{\XtildeN}$ as stated in
Theorem~\ref{thm: ZDT formula} and thus have concluded its proof.

\subsection{GW potentials of $\XtildeN$}\label{subsection: GW
potentials of XtildeN} In this section we prove Corollary~\ref{cor:
formula for F_g}. We begin by computing the genus $g$, fiber curve, GW
potential of $\XtildeN$:
\[
F_{g}^{\XtildeN}(Q,q,y) = \sum_{\begin{smallmatrix} \beta \in H_{2}(\XtildeN)\\
\pi_{*}\beta =0 \end{smallmatrix} } \left\langle \, \,
\right\rangle^{\XtildeN}_{g,\beta}\, \, Q^{\beta \cdot
\Stilde}q^{\beta \cdot \Sprimetilde}y^{\beta \cdot \Deltatilde}.
\]

Our calculation follows closely the computation in
\cite[App~A]{Bryan-Banana}. The \emph{reduced} GW potential $F'_{g}$
is defined by removing the $\beta =0$ term from $F_{g}^{\XtildeN}$:
\[
F'_{g}(Q,q,y) = F^{\XtildeN}_{g}(Q,q,y)-F^{\XtildeN}_{g}(0,0,0).
\]
The GW/DT correspondence conjectured in \cite{MNOP1} and recently
proven by Pardon in \cite{Pardon-Universal-Curve-Counting-2023}, asserts that
\[
\sum_{g=0}^{\infty}F'_{g}(Q,q,y)\lambda^{2g-2} = \log
\left(\frac{Z_{\XtildeN}(p,y,q,Q)}{Z_{\XtildeN}(p,0,0,0)} \right) 
\]
under the change of variables $p=e^{i\lambda}$. Applying this to the
DT partition function of $\XtildeN$ we get
\begin{align}\label{eqn: F'=logZ}
\sum_{g=0}^{\infty}F'_{g}(Q,q,y)\lambda^{2g-2} &= \log
\left(\prod_{k\in \Theta_{N}} \prod_{m,r,s,t}
(1-p^{m}q^{lr}Q^{ks}y^{Nt})^{-c(4rs-t^{2},m)} \right)\\  \nonumber
&= \sum_{k\in \Theta_{N}} \sum_{m,r,s,t}c(4rs-t^{2},m)
\sum_{n=1}^{\infty} \frac{1}{n}p^{nm} q^{nlr}Q^{nks}y^{nNt}
\end{align}
where the indices $(m,r,s,t)$ in the product and sum are given by
integers satisfying $r,s,r+s+t\geq 0$ and $(r,s,t)\neq (0,0,0)$. Now
$c(d,m)=0$ if $d<-1$ \cite[Prop~14]{Bryan-Banana} from which one can
show that an equivalent indexing condition is given by $r,s\geq 0$ and
$t>0$ if $r=s=0$.

In Appendix A of \cite{Bryan-Banana} it is shown that
\begin{equation}\label{eqn: c_{2g-2}(d) coefs in terms of c(d,m)}
\sum_{g=0}^{\infty} c_{2g-2}(d) \lambda^{2g-2} = \sum_{m\in \ZZ}
c(d,m) e^{im\lambda} 
\end{equation}
where $c_{2g-2}(d)$ is defined by
\[
\psi_{2g-2}(q,y) = \sum_{n=0}^{\infty}\sum_{t\in \ZZ}
c_{2g-2}(4n-t^{2})q^{n}y^{t} 
\]
and where $\psi_{2g-2}(q,y)$ is the weak Jacobi form of weight 2g-2
and index 1 given by
\[
\psi_{2g-2}(q,y) = \phi_{-2,1}(q,y)\cdot \begin{cases}
1&g=0\\
\wp(q,y)&g=1\\
\frac{|B_{2g}|}{2g(2g-2)!} E_{2g}(q) &g>1
\end{cases}
\]
(see \cite[App~A]{Bryan-Banana} for further explination). Applying the
substitution $p=e^{i\lambda}$ to Equation~\eqref{eqn: F'=logZ} and
using Equation~\eqref{eqn: c_{2g-2}(d) coefs in terms of c(d,m)} we find
\[
\sum_{g=0}^{\infty}F'_{g}(Q,q,y)\lambda^{2g-2}  = \sum_{k\in
\Theta_{N}} \sum_{r,s,t} \sum_{n=1}^{\infty} \frac{1}{n} q^{nlr}
Q^{nks}y^{nNt} \sum_{g=0}^{\infty} c_{2g-2}(4rs-t^{2}) n^{2g-2}
\lambda^{2g-2}
\]
so that 
\[
F'_{g}(Q,q,y) = \sum_{k\in \Theta_{N}} \sum_{r,s,t}
c_{2g-2}(4rs-t^{2})\operatorname{Li}_{3-2g}(q^{lr}Q^{ks}y^{Nt}) .
\]
For $g>1$ we can add back in the constant terms using for example
\cite[\S~2.1]{MNOP1} to get
\begin{equation}\label{eqn: formula for F_{g} in terms of Li_{3-2g}}
F^{\XtildeN}_{g}(Q,q,y) = \sum_{k\in \Theta_{N}} \left(
c_{2g-2}(0)\cdot \frac{-B_{2g-2}}{4g-4} + \sum_{r,s,t} c_{2g-2}(4rs-t^{2})
\operatorname{Li}_{3-2g}(Q^{ks}q^{lr}y^{Nt}) \right) .
\end{equation}

Let
\[
F^{\ban}_{g}(Q,q,y) = c_{2g-2}(0)\cdot \frac{-B_{2g-2}}{4g-4} + \sum_{r,s\geq 0} \sum_{\begin{smallmatrix} t\in \ZZ \\
t>0 \text{ if }r=s=0 \end{smallmatrix}} c_{2g-2}(4rs-t^{2})
\operatorname{Li}_{3-2g}(Q^{s}q^{r}y^{t})
\]
It is shown in Appendix A of \cite{Bryan-Banana} that
$F^{\ban}_{g}(Q,q,y)$ is the genus 2 Siegel modular form given by the
Maass lift of $\psi_{2g-2}(q,y)$. Then we can rewrite
Equation~\eqref{eqn: formula for F_{g} in terms of Li_{3-2g}} as
\[
F^{\XtildeN}_{g}(Q,q,y) = \sum_{k\in \Theta_{N}}
F_{g}^{\ban}(Q^{k},q^{l},y^{N})
\]
which completes the proof of Corollary~\ref{cor: formula for F_g}.

\begin{proof}[Proof of Corollary \ref{cor: GW potential of Xtilde is a paramodular form}]

Let us define the function 
\[
\mathcal{F}^{\XtildeN}_{g}(Q, q, y) = \sum_{k \in \Theta_{N}} F^{\ban}_{g}(Q^{Nk}, q^{\frac{N}{k}}, y^{N})
\]
which is clearly related to the GW potential $F^{\XtildeN}_{g}$ through the simple change of variables $Q \mapsto Q^{N}$. We will prove Corollary \ref{cor: GW potential of Xtilde is a paramodular form} by establishing the automorphic properties of $\mathcal{F}^{\XtildeN}_{g}$.

\begin{lemma}
If $F(Q, q, y)$ is a Siegel modular form on $\Sp_{4}(\ZZ)$, then for $N$ and $k$ positive integers, $F(Q^{Nk}, q^{\frac{N}{k}}, y^{N})$ is a Siegel modular form of the same weight for the subgroup 
\[
L_{N,k} = \Sp_{4}(\QQ) \cap
\begin{pmatrix}
\ZZ & k \, \ZZ & \tfrac{k}{N} \, \ZZ & \tfrac{1}{N} \, \ZZ \\
\frac{1}{k} \, \ZZ & \ZZ & \frac{1}{N} \, \ZZ & \frac{1}{Nk} \, \ZZ \\
\frac{N}{k} \, \ZZ  & N \, \ZZ  & \ZZ & \frac{1}{k} \, \ZZ  \\
N \, \ZZ & Nk \, \ZZ & k \, \ZZ & \ZZ 
\end{pmatrix}.
\]
\end{lemma}

\begin{proof}
Recall that $Q=e^{2 \pi i \sigma}$, $q=e^{2 \pi i \tau}$, $y=e^{2 \pi i z}$ where 
$\Omega = 
\left(\begin{smallmatrix} \tau &z\\ 
z&\sigma \end{smallmatrix} \right)$ is an element of the genus $2$ Siegel upper half space $\HH_{2}$. The diagonal matrix
\[
g = \diag(N, Nk, k, 1)
\]
acts on $\HH_{2}$ in the standard way by
\[
\Omega \mapsto g \cdot \Omega \\
=
\begin{pmatrix}
N & 0 \\
0 & Nk 
\end{pmatrix} \Omega
\begin{pmatrix}
k & 0 \\
0 & 1 
\end{pmatrix}^{-1} = 
\begin{pmatrix}
\frac{N}{k}\tau & Nz \\
Nz & Nk \sigma 
\end{pmatrix}.
\]
If $F(\Omega) = F(Q, q, y)$ is a Siegel modular form on $\Sp_{4}(\ZZ)$, then $F(g \cdot \Omega) = F(Q^{Nk}, q^{\frac{N}{k}}, y^{N})$ is a Siegel modular form of the same weight on $\Sp_{4}(\QQ) \cap (g^{-1} \Sp_{4}(\ZZ) g)$. For the particular matrix $g$ above, it is straightforward to verify that
\[
L_{N,k} = \Sp_{4}(\QQ) \cap (g^{-1} \Sp_{4}(\ZZ) g).
\]
\end{proof}

Recalling that $F^{\ban}_{g}(Q, q, y)$ is a Siegel modular form of
weight $2g-2$ on $\Sp_{4}(\ZZ)$, the above lemma implies that
$\mathcal{F}^{\XtildeN}_{g}$ is a Siegel modular form of weight $2g-2$
for the group  
\[
\cap_{k \in \Theta_{N}} L_{N,k}
\]
which one can easily show is exactly the subgroup $P_{N}\subset
\Sp_{4}(\QQ)$ defined by Equation~\eqref{eqn: defn of subgroup
PN}. Finally, we note that by the standard action of $\Sp_{4}(\RR)$ on
$\HH_{2}$, the involution $\iota_{N}$ (defined in equation~\eqref{eqn:
defn of involution iN}) induces the transformation
\[
(Q, q, y) \mapsto (q^{\frac{1}{N}}, Q^{N}, y)
\]
under which $\mathcal{F}^{\XtildeN}_{g}$ is evidently invariant. It
follows that $\mathcal{F}^{\XtildeN}_{g}$ is a Siegel modular form for
the index $2$ normal extension $P^{*}_{N} \subset \Sp_{4}(\RR)$, which
completes the proof of Corollary \ref{cor: GW potential of Xtilde is a
paramodular form}.

\end{proof}

\subsection{GV invariants of $\XtildeN$}

In \cite{Bryan-Banana} it is shown that if $a,b,c\subset \Fhat_{\ban}$
are banana curves the the GV invariants $n^{g}_{\beta}(\Fhat_{\ban})$
of an effective class 
\[
\beta =d_{1}a + d_{2} b + d_{3} c
\]
only depend on $g$ and the quantity
\[
a= \dvec^{2} = 2d_{1}d_{2}+ 2d_{2}d_{3}+ 2d_{3}d_{1} - d_{1}^{2} - d_{2}^{2} - d_{3}^{2}
\]
and then $n^{g}_{\beta}(\Fhat_{\ban})=n^{g}_{a}$ where the integers
$n^{g}_{a}$ are given by the formula
\[
\sum_{a=-1}^{\infty}\sum_{g=0}^{\infty} n_{a}^{g}\,
(y^{\half}+y^{-\half})^{2g}q^{a+1} =  \prod_{n=1}^{\infty}
\frac{(1+yq^{2n-1})(1+y^{-1}q^{2n-1})
(1-q^{2n})}{(1+yq^{4n})^{2}(1+y^{-1}q^{4n})^{2}(1-q^{4n})^{2} }.
\]

Now consider an effective curve class on $\Fhat (k)\subset \XtildeN$
given by 
\[
\beta  = d_{1}\atilde (k)+ d_{2}\btilde (k)+ d_{3}\ctilde (k).
\]
Recall that we have a quadratic form $||\cdot ||$ on fiber curve
classes induced from the intersection form on a smooth fiber. We
compute $||\beta ||$ as follows:
\begin{align*}
||\beta || &= ||d_{1}(\atilde(k) +\ctilde(k) )+d_{2}(\btilde(k) +\ctilde(k) )+ (d_{3}-d_{1}-d_{2})\ctilde (k)||\\
&= ||d_{1}\cdot \frac{1}{k}\cdot \ftilde + d_{2}\cdot \frac{1}{l}\cdot
\fprimetilde + (d_{3}-d_{1}-d_{2})\cdot \frac{-1}{2N}\cdot \deltatilde ||\\
&= 2d_{1}d_{2}\cdot \frac{1}{kl}\cdot N + (d_{3}-d_{1}-d_{2})^{2}\cdot
\frac{1}{4N^{2}}\cdot (-2N^{2}) \\
&=\frac{1}{2}\left(4d_{1}d_{2}-(d_{3}-d_{1}-d_{2})^{2} \right)\\
&=\frac{1}{2}\left( 2d_{1}d_{2}+ 2d_{2}d_{3}+ 2d_{3}d_{1} - d_{1}^{2}
- d_{2}^{2} - d_{3}^{2} \right) .
\end{align*}
And thus if we let $a=2||\beta ||$ then $n^{g}_{\beta}(\Fhat
(k))=n^{g}_{a}$.

Then for a general effective fiber curve class $\beta$ on $\XtildeN$
with $a=2||\beta ||$ we have
\begin{equation}\label{eqn: ngbeta = epsilonNbeta * nga}
n^{g}_{\beta}(\XtildeN ) = \epsilon_{N}(\beta ) n^{g}_{a}
\end{equation}
where $\epsilon_{N}(\beta )$ is the number of singular fibers
$F(k)\subset \XtildeN$ on which $\beta$ is represented by an integral class.

For example the class
\begin{align*}
\atilde (N) + \ctilde (N) &= \frac{1}{N}\cdot \ftilde \\
&= \frac{k}{N} (\atilde (k) + \ctilde (k))
\end{align*}
is represented by an effective curve in $F(N)$, but is not represented
by a curve in $F(k)$ when $k\neq N$, so in this case, $\epsilon_{N}\left(\atilde
(N)+\ctilde (N) \right)=1$.

\begin{lemma}\label{lem: formula for epsilonN(beta)}
\[
\epsilon_{N}(\beta ) = \sum_{k\in \Theta_{N}} \epsilon_{N,k}(\beta )
\]
where
\[
\epsilon_{N,k}(\beta ) = \begin{cases}
1&\text{if $k\divides (\beta \cdot \Stilde)$ and $l\divides (\beta \cdot \Sprimetilde)$}\\
0&\text{otherwise.}
\end{cases}
\]
\end{lemma}

\begin{proof}
An integral class on $F(k)$ given by 
\[
\beta =d_{1}\atilde (k)+d_{2}\btilde (k)+d_{3}\ctilde (k)
\]
can be written as
\[
\beta =d_{1}\cdot \frac{1}{k}\cdot \ftilde +d_{2}\cdot
\frac{1}{l}\cdot \fprimetilde  + (d_{3}-d_{1}-d_{2})\cdot \frac{-1}{2N}\cdot \deltatilde 
\]
and thus satisfies
\[
\beta \cdot \Sprimetilde =d_{1}l,\quad \beta \cdot \Stilde
=d_{2}k,\quad \beta \cdot \Deltatilde =(d_{3}-d_{1}-d_{2})N
\]
and so in particular $k$ divides $ \beta \cdot \Stilde$ and
$l$ divides  $ \beta \cdot \Sprimetilde$. Moreover, $N$ divides $\beta \cdot
\Deltatilde$ for any effective fiber curve class.

Conversely, suppose that $\beta$ is an effective fiber curve class
satisfying  $k\divides (\beta \cdot \Stilde)$ and $l\divides
(\beta \cdot \Sprimetilde)$. Then $\beta$ is represented by an
integral curve class on $F(k)$ since we may define the integers
\[
d_{1}=\frac{1}{l}\beta \cdot \Sprimetilde ,\quad d_{2} = \frac{1}{k}\beta \cdot \Stilde 
\]
and then 
\[
d_{3} = \frac{1}{N}\beta \cdot \Deltatilde +d_{1}+d_{2}.
\]
\end{proof}
Lemma~\ref{lem: formula for epsilonN(beta)} and Equation~\eqref{eqn:
ngbeta = epsilonNbeta * nga} then complete the proof of
Proposition~\ref{prop: formula for GV invariants of XN}.

\appendix
\section{Rigid CY3s related by finite quotients and small
resolutions (with Mike Roth) }\label{sec: reduction to Verills case}

Recall that for a rigid CY3 $X$ defined over $\QQ$ there exists a
weight 4 modular cusp form
\[
f_{X}(q) = \sum_{n=1}^{\infty} a_{n}q^{n}
\] 
uniquely characterized by the condition that 
\begin{equation*}
a_{p} =\Tr (\Frob_{p} \divides H^{3}_{\et}(X_{\overline{\FF}_{p}},\QQ_{l}))
\end{equation*}
for almost all primes $p$.

In \cite[Appendix]{Yui-with-Verrill-appendix} Verrill considered rigid
CY3s which are closely related to our banana nano-manifolds
$\XtildeN$. Namely, let $X^{\Ver}_{N}$ be (any) projective conifold
resolution of the fiber product $S_{N}\times_{\PP^{1}}S_{N}$ (Schoen
\cite{Schoen-88} proved that there exists projective resolutions of
any self-fiber product of a rational surface). Then $X^{\Ver}_{N}$ is
a rigid CY3 with $h^{1,1}(X^{\Ver}_{N})=\sum_{k\in
\Theta_{N}}k^{2}$. Verrill uses a particular model for $S_{N}$ which
is defined over $\QQ$ and proves
that
\[
f_{X^{\Ver}_{N}}(q) = \prod_{k\in \Theta_{N}} \eta(q^{k})^{2} 
\]
independent of the choice of the conifold resolution\footnote{The
existence of a projective conifold resolution follows from a theorem
of Schoen. However, Schoen's argument does not guarantee that the
conifold resolution is defined over $\QQ$. We will address this issue
in Lemma~\ref{lem: quotient and resolution maps are defined over
Q}.}
\[
X^{\Ver}_{N}\xrightarrow{\pi_{1}}S_{N}\times_{\PP^{1}}S_{N}.
\]
We note that the quotient of $S_{N}\times_{\PP^{1}}S_{N}$ by $G_{N}$
acting on the second factor is 
\[
S_{N}\times_{\PP^{1}}S_{\sing} = X_{\sing}.
\]
Thus $\XtildeN$ and $X^{\Ver}_{N}$ are related by the following
sequence of maps
\begin{equation}\label{eqn: XVer->...->XtildeN}
X^{\Ver}_{N}\xrightarrow{\pi_{1}}S_{N}\times_{\PP^{1}}S_{N}
\xrightarrow{q_{1}} X_{\sing} \xleftarrow{\pi_{2}} X_{N}
\xrightarrow{q_{2}} \XtildeN 
\end{equation}
where the maps $\pi_{1}$ and $\pi_{2}$ are crepant resolutions and
$q_{1}$ and $q_{2}$ are both quotients by the action of the finite
group $G_{N}$. 

We will show that the above maps and varieties are defined over $\QQ$
and that they induce an isomorphism
\[
H^{3}_{\et}(X^{\Ver}_{N},\QQ_{l}) \cong H^{3}_{\et}(\XtildeN ,\QQ_{l})
\]
which is compatible with the action of Frobenius so that in particular
\[
f_{X^{\Ver}_{N}}(q) = f_{\XtildeN}(q).
\]
Let $V$ be a variety defined over $\CC$.  We recall that the both the
ordinary cohomology groups $H^{i}(V,\QQ)$, and the groups
$H_{c}^{i}(V,\QQ)$ of cohomology with compact support carry weight
filtrations, increasing filtrations $F_0\subseteq F_1 \subseteq F_2
\subseteq \cdots$ of $\QQ$-subspaces.  We denote by
$\Gr_{m}H^{i}(V,\QQ)$ (or $\Gr_{m}H_{c}^{i}(V,\QQ)$) the quotient
$F_{m}/F_{m-1}$.   Given a map $\pi\colon V\longrightarrow V'$ of
varieties, the pullback maps on cohomology are compatible with the
weight filtrations, and so induce maps of the graded pieces.

Moreover, if one has an exact sequence of cohomology groups, for
instance an excision sequence 
\[
\cdots \longrightarrow H^{i-1}(Z,\QQ)
\longrightarrow H^{i}_{c}(U,\QQ) \longrightarrow H^{i}(V,\QQ)
\longrightarrow H^{i}(Z,\QQ) \longrightarrow\cdots
\]
then for any $m$
the sequence of graded pieces 
\[
\cdots \longrightarrow
\Gr_{m}H^{i-1}(Z,\QQ) \longrightarrow \Gr_{m}H^{i}_{c}(U,\QQ)
\longrightarrow \Gr_{m}H^{i}(V,\QQ) \longrightarrow
\Gr_{m}H^{i}(Z,\QQ) \longrightarrow\cdots
\]
is again exact.  This is a
consequence of the fact that the maps on cohomology are not only
compatible with the filtrations, but are strictly compatible
\cite[Proposition 1.1.11]{Deligne-Hodge-II}.

\newcommand{\AlphaList}{\renewcommand{\labelenumi}{{({\em\alph{enumi}})}}}
\renewcommand{\labelenumii}{({\em \alph{enumi}\Old{\arabic{enumii}}})}
\AlphaList
\newcommand{\Old}[1]{\oldstylenums{#1}}

\begin{lemma}\label{lem: resolutions V->V' with some exceptional sets induces iso on Gr3H3}

Let $V$ and $V'$ be projective threefolds defined over $\CC$ and
$\pi\colon V\longrightarrow V'$ a birational map.  Let $Z\subset V$ be
the exceptional locus of $\pi$ and $Z'=\pi(Z)$. 

\begin{enumerate} \item Suppose that $Z$ and $Z'$ have the property
that $\Gr_{3}H^2(Z,\QQ)=\Gr_{3}H^3(Z,\QQ)=0$ and that
$\Gr_{3}H^2(Z',\QQ)=\Gr_{3}H^3(Z',\QQ)=0$.  Then the pullback map
$\pi^{*}\colon H^3(V',\QQ)\longrightarrow H^3(V,\QQ)$ induces an
isomorphism on $\Gr_3$. 

\medskip \item The vanishing conditions in (a) hold in each of the
following cases:

\begin{enumerate} 
\item For the morphisim 
$\pi_2\colon X_{N}\longrightarrow X_{\sing}$ above;
\item When $\pi\colon
V\longrightarrow V'$ is a conifold resolution 
(e.g., 
$\pi_1\colon X^{\Ver}_{N}\longrightarrow S_{N}\times_{\PP^{1}}S_{N}$).
\end{enumerate}
\end{enumerate}
\end{lemma}

\begin{proof}

Let $$U=V-Z,\rule{0.25cm}{0cm} U'=V'-Z',$$ and note that $\pi$ induces an isomorphism $U\cong U'$.

The excision exact sequences for $U=V-Z$ and $U'=V'-Z'$ are compatible with the maps $\pi^{*}$ on cohomology
and lead to the following commutative diagram with exact rows:
\[
\begin{tikzcd}
H^{2}(Z',\QQ) \arrow{r} \arrow{d}{\pi^{*}} &
H^{3}_{c}(U',\QQ) \arrow{r}{i} \arrow[d,"\wr"',"\pi^{*}"] &
H^{3}(V',\QQ) \arrow{r}{} \arrow{d}{\pi^{*}} & H^{3}(Z',\QQ) \arrow{d}{\pi^{*}}   \\
H^{2}(Z,\QQ) \arrow{r}{} & H^{3}_{c}(U,\QQ)
\arrow{r}{} & H^{3}(V,\QQ) \arrow{r}{} & H^{3}(Z,\QQ)
\end{tikzcd}
\]
Passing to $\Gr_3$ and using the vanishing hypotheses, this diagram 
becomes 
\[
\begin{tikzcd}
0\arrow{r} &
H^{3}_{c}(U',\QQ) \arrow{r}{i} \arrow[d,"\wr"',"\pi^{*}"] &
\Gr_{3}H^{3}(V',\QQ) \arrow{r}{} \arrow{d}{\pi^{*}} & 0 \\
0 \arrow{r}{} & \Gr_{3}H^{3}_{c}(U,\QQ)
\arrow{r}{} & \Gr_{3}H^{3}(V,\QQ) \arrow{r}{} & 0
\end{tikzcd}
\]
Consequently the map $\Gr_{3}H^{3}(V',\QQ)\longrightarrow \Gr_{3}H^{3}(V,\QQ)$ is an isomorphism.

We next check the vanishing conditions for $\pi_2\colon
X_{N}\longrightarrow X_{\sing}$.  In this case $Z'\subset X_{\sing}$
is given by the union of the curves $I_{k}\times n'$ where
$I_{k}\subset S$ are the singular fibres and $n'\in S'_{\sing}$ is the
nodal point in the corresponding fibre.  Since $Z'$ is complex
one-dimensional, $H^3(Z',\QQ)=0$.  Since $Z'$ is proper, all weights
of $H^{i}(Z',\QQ)$ are $\leq i$, by \cite[Th\'{e}or\`{e}me
8.2.4]{Deligne-Hodge-III}.  In particular all weights of $H^2(Z',\QQ)$
are $\leq 2$ and so $\Gr_{3}H^2(Z',\QQ)=0$.

The exceptional locus $Z\rightarrow Z'$ is a union of components in
the singular fibres of $X_{N}$ which by the local toric description of
Section~\ref{sec: local toric geometry} is a normal crossing divisor.
The class of $Z$ in the Grothendieck group of varieties is a
polynomial in $\mathbb{L}=[\AA^{1}]$, the class of the affine line. It
follows that the weight polynomial of $Z$ is supported in even degrees
and in particular, we have that $\Gr_{3}H^{3}(Z)=0$.  Since $Z$ is
proper, $\Gr_{3}H^2(Z,\QQ)=0$ as above.

Finally, we check the vanishing conditions for a conifold resolution.
Here $Z'$ is a finite set of points and so has no cohomology above
$H^0$, and $Z$ is a curve, and so the degree $3$ parts of $H^2$ and
$H^3$ vanish as in the previous case.
\end{proof}

\begin{lemma}\label{lem: quotient and resolution maps are defined over
Q}
The morphisms and varieties given in equation~\eqref{eqn:
XVer->...->XtildeN} are all defined over $\QQ$.
\end{lemma}
\begin{proof}
As we make explicit in Appendix~\ref{appendix: the surfaces SN}, both
Verrill's model for $S_{N}\to \PP^{1}$ and the group action
$G_{N}\times S_{N}\to S_{N}$ are defined over $\QQ$. Consequently,
$S_{N}\times_{\PP^{1}}S_{N}$,
$X_{\sing}=S_{N}\times_{\PP^{1}}(S_{N}/G_{N})$ and the morphism
$q_{1}$ are all defined over $\QQ$. The map $q_{1}$ is the composition
of $X_{N}\to X_{\con} $ and $X_{\con}\to X_{\sing}$. This later map is
induced by the minimal resolution $S'_{N}\to S_{N}/G_{N}$ on the
second factor which is defined over $\QQ$ because we may take $S'_{N}$
to be the so-called $G$-Hilbert scheme which is a component of the
Hilbert scheme of substacks of the stack quotient $[S_{N}/G_{N}]$ and
the morphism to be the Hilbert-Chow morphism (the Hilbert scheme of
substacks of a stack defined over $\QQ$ is defined over $\QQ$). The
morphism $X_{N}\to X_{\con}$ is given by the blowup of $\Gamma$ which
is a divisor defined over $\QQ$ since it is the proper transform of
the graph of a morphism defined over $\QQ$.

Finally, we need to see that Verrill's conifold resolution
$X_{N}^{\Ver}\to S_{N}\times_{\PP^{1}}S_{N}$ is defined over $\QQ$
(this issue does not appear to be addressed in the original
paper). Schoen proves the existence of a projective resolution of any
$S\times_{\PP^{1}}S$ self-product of a rational elliptic surface with
singular fibers of $I_{n}$ type and Verrill quotes this
result. However, Schoen's argument (first blow up the diagonal and
then successively blow up irreducible components of the singular
fibers) does not guarentee that the result is defined over $\QQ$ since
components of the singular fibers may not be defined over $\QQ$
(indeed they are not in general in our case). We may nevertheless find
a conifold resolution defined over $\QQ$ as follows. The diagonal, the
$G_{N}$-orbits of the diagonal, and $H$-orbits of the diagonal for
subgroups $H\subset G_{N}$ are all Weil divisors defined over
$\QQ$. We can obtain a projective conifold resolution of
$S_{N}\times_{\PP^{1}}S_{N}$ defined over $\QQ$ by successively
blowing up those Weil divisors and their proper transforms in various
orders. The specifics of this process depend on $N$, which is not hard
to determine with explicit analysis of the singular
fibers. Explicitly, for $N=5$ it suffices to first blowup the diagonal
and then blowup the proper transform of the $\ZZ_{5}$ orbit of the
diagonal. For $N=6$, first blow up the diagonal, then blow up the
proper transform of the $\ZZ _{2}$-orbit of the diagonal, then blowup
the proper transform of the $\ZZ_{3}$-orbit of the diagonal. For
$N=8$, first blowup the $\ZZ_{2}$-orbit of the diagonal and then
blowup the proper transform of the $\ZZ_{4}$-orbit of the
diagonal. For $N=9$, blowing up the full $\ZZ_{3}\times \ZZ_{3}$ orbit
of the diagonal works.
\end{proof}

We will also use the following standard result which may be easily
proved using the Leray-Serre spectral sequence.
\begin{lemma}\label{lem: H^i(V/G)=H^i(V)^G}
Let $G$ be a finite group acting on $V$. Then $H^{i}(V/G,\QQ)\cong
H^{i}(V,\QQ)^{G}$ where the inclusion $H^{i}(V/G,\QQ)\cong
H^{i}(V,\QQ)^{G}\hookrightarrow H^{i}(V,\QQ)$ is given by $p^{*}$ 
where $p:V\to V/G$.
\end{lemma}

We now complete the proof of Theorem~\ref{thm: theorem on weight four
cusp forms}. 
We examine the maps on cohomology induced by \eqref{eqn:
XVer->...->XtildeN} :
\[
H^{3}(X^{\Ver}_{N},\QQ) \xleftarrow{\pi_{1}^{*}}
H^{3}(S_{N}\times_{\PP^{1}}S_{N},\QQ) \xhookleftarrow{q_{1}^{*}}
H^{3}(X_{\sing},\QQ)\xrightarrow{\pi^{*}_{2}} H^{3}(X_{N},\QQ)
\xhookleftarrow{q_{2}^{*}} H^{3}(\XtildeN,\QQ )
\]
Restricting to the weight 3 graded piece of the above and using the
fact that $X^{\Ver}_{N}$, $X_{N}$, and $\XtildeN$ are non-singular
projective threefolds we get 
\[
H^{3}(X^{\Ver}_{N},\QQ) \xleftarrow{\pi_{1}^{*}}
\Gr_{3}H^{3}(S_{N}\times_{\PP^{1}}S_{N},\QQ) \xhookleftarrow{q_{1}^{*}}
\Gr_{3}H^{3}(X_{\sing},\QQ)\xrightarrow{\pi^{*}_{2}} H^{3}(X_{N},\QQ)
\xhookleftarrow{q_{2}^{*}} H^{3}(\XtildeN,\QQ)
\]
By Lemma~\ref{lem: resolutions V->V' with some exceptional sets induces iso on Gr3H3}
$\pi^{*}_{2}$ and $\pi^{*}_{1}$ are
isomorphisms on the degree $3$ pieces. By Lemma~\ref{lem: H^i(V/G)=H^i(V)^G}, 
$q_{1}^{*}$ and $q_{2}^{*}$ are both injective. 
Since $X^{\Ver}_{N}$ and $\XtildeN$
are both rigid CY3s, $H^{3}(X^{\Ver}_{N},\QQ)$ and 
$H^{3}(\XtildeN,\QQ)$ are
both isomorphic to $\QQ\oplus \QQ$ and hence the injection
$(\pi_{1}^{*})^{-1}\circ q_{1}^{*}\circ (\pi_{2}^{*})^{-1}\circ
q_{2}^{*}$ is an isomorphism
\[
H^{3}(\XtildeN,\QQ )\cong H^{3}(X^{\Ver}_{N},\QQ).
\]

Fix a prime $l$.  Then by the comparison theorem \cite[Lecture~11, Theorem~4.4]{SGA4}
$H^{3}_{\et}(\tilde{X}_{N},\QQ_{l}) \cong
H^{3}(\tilde{X}_{N},\QQ)\otimes_{\QQ}\QQ_{l}$ and 
$H^{3}_{\et}(X_{N}^{\Ver},\QQ_{l}) 
\cong H^{3}(X_{N}^{\Ver},\QQ)\otimes_{\QQ}\QQ_{l}$.
The isomorphisms provided by the comparision theorem are compatible 
with pullbacks, and so the map
$(\pi_{1}^{*})\circ q_1^{*}\circ(\pi_2^{*})^{-1}\circ q_{2}^{*}$ also induces an isomorphism 
$$H^3_{\et}(\tilde{X}_{N},\QQ_{l}) \cong H^3_{\et}(X_{N}^{\Ver},\QQ_{l}).$$

As a consequence of Lemma~\ref{lem: quotient and resolution maps are defined over
Q}, the maps $\pi_1$, $\pi_2$, $q_1$, and $q_2$ are all defined over 
$\QQ$, and so this isomorphism on cohomology groups is also an 
isomorphism of $\Gal(\overline{\QQ}/\QQ)$ representations. 
Thus we have the equality 
$$f_{\tilde{X}_{N}}(q) = f_{X^{\mbox{\scriptsize\sf Ver}}_{N}}(q).$$
This completes the proof of Theorem~\ref{thm: theorem
on weight four cusp forms}.

\bigskip

It remains to prove Propostion~\ref{prop: f_E(q)^2 = f_XN(q^2)}. The
four singular fibers of $S_{N} \to \PP^{1}$ occur at points
$p_{1},\dotsc ,p_{4}\in \PP^{1}$ which are given explicity in
\cite[Table 2]{Yui-with-Verrill-appendix}. In all cases, $p_{1} =
\infty$, and the cross-ratio of the four points is given by
\[
\lambda = \frac{p_{3}-p_{2}}{p_{3}-p_{4}}
\]
If $E_{N}$ is the double cover of $\PP^{1}$ branched at $\{p_{1}, \ldots, p_{4}\}$, then the $j$-invariant of $E_{N}$ is given by
\[
j(\lambda) = 2^{8} \frac{(\lambda^{2}-\lambda+1)^{3}}{\lambda^{2}(1-\lambda)^{2}}
\]
Thus, one can compute $j(\lambda)$ in each case, and use the LMFDB
\cite{LMFDB} to find a suitable model of $E_{N}$ over $\QQ$ whose
corresponding weight $2$ cusp modular form $f_{E_{N}}(q)$ satisfies
$f_{E_{N}}(q)^{2} = f_{\XtildeN}(q^{2})$. The data is presented in the
following table, together with the appropriate LMFDB labels:

\vskip2ex

\begin{center}
\begin{tabular}{lllll}\label{table: LMFDB Data}
$N$&  $j(\lambda)$  &	$\QQ$-model of $E_{N}$ (LMFDB)  &  Weierstrass form of $\QQ$-model & $f_{E_{N}}(q)$ \\[3pt]
\hline\\[-8pt]
$5$&	  $\frac{488095744}{125}$ & 20.a1&	$y^{2}=x^{3}+x^{2}-41x-116$ & $\eta(q^{10})^{2} \eta(q^{2})^{2}$ \\[3pt]
$6$&	  $ \frac{1556068}{81}$ & 24.a3&   $y^{2}=x^{3}-x^{2}-24x-36$	& $\eta(q^{12})\eta(q^{6}) \eta(q^{4}) \eta(q^{2})$ \\[3pt]
$8$&	  $1728$  & 32.a3&  $y^{2}=x^{3}-x$ &  $\eta(q^{8})^{2} \eta(q^{4})^{2}$	\\[3pt]
$9$&	  $0$  & 36.a3 &   $y^{2} = x^{3}-27$	 & $\eta(q^{6})^{4}$ \\[3pt]

\end{tabular}
\end{center}

\section{Explicit $\QQ$-models for the surfaces $S_{N}$}\label{appendix: the surfaces SN}

For reference, we include some basic explicit data for the models over $\QQ$ of the surfaces $S_{N}$ studied by Beauville \cite{Beauville-Elliptic-surface-with-4-singularfibers} and
by Verrill \cite{Yui-with-Verrill-appendix}. In this model, the surface $S_{N}$ is given by the minimal resolution
of a hypersurface $\overline{S}_{N}\subset
\PP^{2}\times \PP^{1}$ with equation $f_{N}(x,y,z,\lambda ,\mu )=0$,
which is homogeneous of degree 3 and 1 in the variables $(x,y,z)\in \PP^{2}$
and  $(\lambda ,\mu )\in \PP^{1}$ respectively. 

In each case the Mordell-Weil group $G_{N}$ is finite, and we include here explicit
equations for the group action $G_{N}\times S_{N}\to S_{N}$ (which to
our knowledge does not appear anywhere else in the
literature). Departing notationally from Beauville and Verrill, we
choose to index the cases by the order $N$ of the Mordell-Weil
group. Note that the $N=3$ and $N=4$ cases do not arise in the main
body of the paper, as they do not lead to a construction of a banana
nano-manifold. 

\vskip2ex  

\begin{center}
\begin{tabular}{llll}\label{table: SN equations}
$N$&	$f_{N}$&	$G_{N}$&	Generator(s) for the $G_{N}$ action \\[3pt]
\hline\\[-8pt] 
$3$&	$\mu (x^{2}y+y^{2}z+z^{2}x) - \lambda xyz$  & $\ZZ_{3}$&
$(x,y,z)\mapsto (y,z,x)$\\[3pt] 

$4$&	$\mu (x+y)(xy-z^{2}) -\lambda xyz$& $\ZZ_{4}$&	$(x,y,z)\mapsto (xy,-z^{2},xz)$\\[3pt]	
 
$5$&	$\mu x(x-y)(y-z) -\lambda xyz$& $\ZZ_{5}$&
$(x,y,z)\mapsto (y(x-z),-yz,z(x-y))$\\[3pt]	 

$6$&	$\mu (x+y+z)(xy+yz+zx) - \lambda xyz$& $\ZZ_{6}$&
$(x,y,z)\mapsto (xy,yz,xz)$\\[3pt]	 

$8$&	$\mu x(x^{2}+z^{2}+2zy)-\lambda z(x^{2}-y^{2})$&
$\ZZ_{4}\times \ZZ_{2}$ &  \text{(See below)}\\[3pt] 
 
$9$&	$\mu (x^{3}+y^{3}+z^{3})- \lambda xyz$& $\ZZ_{3}\times
\ZZ_{3}$&	$(x,y,z)\mapsto (y,z,x)$\\[3pt] 
 &	&&	$(x,y,z)\mapsto (x,\omega y,\omega^{2}z),\quad \omega^{3}=1$\\ 

\end{tabular}
\end{center}

\vskip2ex

\noindent In the case of $N=8$, the following maps of order $4$ and
$2$, respectively, generate the $\ZZ_{4} \times \ZZ_{2}$ action 
\begin{equation*}
\begin{split}
& {\scriptstyle (x,y,z) \,\,\, \mapsto \,\,\, \big(
(x-y)(x-z)^2(x^2+z(2y+z)),\,\,\,\,
(x-y)(x-z)(x^3-x^2z+xz(2y+z)+z(2y^2+2yz+z^2), \,\,\,\,
-(x+y)(x^2+z(2y+z))^2)\big)} \\[4pt] 
& {\scriptstyle (x,y,z) \,\,\, \mapsto \,\,\, \big(x
(y+z)(x^{2}+yz)^{2}(x^{2}+z(2y+z)), \,\,\,\, (x^{2}+yz)(x^{6}
+3x^{4}yz + y^{3}z^{3} + x^{2}z(y^{3}+6y^{2}z+3yz^{2} + z^{3})),
\,\,\,\,  -x^{2}z(y+z)^{3}(x^{2} + z(2y+z)) \big)}  
\end{split}
\end{equation*}

In the case of $N \in \{3, 6, 9\}$, the generators of the action can
be determined by inspection. The remaining cases require a
straightforward calculation using the group law of a generic fiber of
$\overline{S}_{N} \to \PP^{1}$. This smooth cubic curve intersects its
Hessian curve in $9$ inflection points, one of which can be chosen as
the origin. An analysis of the cubic pencil $f_{N} = 0$ determines the
sections, and the translation morphism by a given section can thus be
determined from the group law as a birational automorphism of
$\PP^{2}$.

\bibliography{localbiblio.bib}

\begin{thebibliography}{10}

\bibitem{SGA4}
{\em Th\'{e}orie des topos et cohomologie \'{e}tale des sch\'{e}mas. {T}ome 3},
  volume Vol. 305 of {\em Lecture Notes in Mathematics}.
\newblock Springer-Verlag, Berlin-New York, 1973.
\newblock S\'{e}minaire de G\'{e}om\'{e}trie Alg\'{e}brique du Bois-Marie
  1963--1964 (SGA 4), Dirig\'{e} par M. Artin, A. Grothendieck et J. L.
  Verdier. Avec la collaboration de P. Deligne et B. Saint-Donat.

\bibitem{Beauville-Elliptic-surface-with-4-singularfibers}
Arnaud Beauville.
\newblock Les familles stables de courbes elliptiques sur {${\bf P}^{1}$}
  admettant quatre fibres singuli\`eres.
\newblock {\em C. R. Acad. Sci. Paris S\'{e}r. I Math.}, 294(19):657--660,
  1982.

\bibitem{Behrend-micro}
Kai Behrend.
\newblock Donaldson-{T}homas type invariants via microlocal geometry.
\newblock {\em Ann. of Math. (2)}, 170(3):1307--1338, 2009.
\newblock arXiv:math/0507523.

\bibitem{Bryan-Banana}
Jim Bryan.
\newblock The {D}onaldson-{T}homas partition function of the banana manifold.
\newblock {\em Algebr. Geom.}, 8(2):133--170, 2021.
\newblock With an appendix coauthored with Stephen Pietromonaco.
  arXiv:math/1902.08695.

\bibitem{Bryan-Oberdieck-CHL}
Jim Bryan and Georg Oberdieck.
\newblock C{HL} {C}alabi-{Y}au threefolds: curve counting, {M}athieu moonshine
  and {S}iegel modular forms.
\newblock {\em Commun. Number Theory Phys.}, 14(4):785--862, 2020.
\newblock arXiv:math/1811.06102.

\bibitem{Candelas-CY3-small-Hodge}
Philip Candelas, Andrei Constantin, and Challenger Mishra.
\newblock Calabi-{Y}au threefolds with small {H}odge numbers.
\newblock {\em Fortschr. Phys.}, 66(6):1800029, 21, 2018.
\newblock arXiv:1602.06303.

\bibitem{Deligne-Hodge-II}
Pierre Deligne.
\newblock Th\'{e}orie de {H}odge. {II}.
\newblock {\em Inst. Hautes \'{E}tudes Sci. Publ. Math.}, (40):5--57, 1971.

\bibitem{Deligne-Hodge-III}
Pierre Deligne.
\newblock Th\'{e}orie de {H}odge. {III}.
\newblock {\em Inst. Hautes \'{E}tudes Sci. Publ. Math.}, (44):5--77, 1974.

\bibitem{Dieulefait-Modularity-of-CY3s}
Luis Dieulefait.
\newblock On the modularity of rigid {C}alabi-{Y}au threefolds: epilogue.
\newblock {\em Zap. Nauchn. Sem. S.-Peterburg. Otdel. Mat. Inst. Steklov.
  (POMI)}, 377:44--49, 241, 2010.

\bibitem{Gouvea-Yui-rigidCY3s-are-modular}
Fernando~Q. Gouv\^{e}a and Noriko Yui.
\newblock Rigid {C}alabi-{Y}au threefolds over {$\Bbb Q$} are modular.
\newblock {\em Expo. Math.}, 29(1):142--149, 2011.
\newblock arXiv:math/0902.1466.

\bibitem{Gritsenko99}
V.~Gritsenko.
\newblock Elliptic genus of {C}alabi-{Y}au manifolds and {J}acobi and {S}iegel
  modular forms.
\newblock {\em Algebra i Analiz}, 11(5):100--125, 1999.
\newblock arXiv:math/9906190.

\bibitem{Kanazawa-Lau}
Atsushi Kanazawa and Siu-Cheong Lau.
\newblock Local {C}alabi-{Y}au manifolds of type {$\tilde A$} via {SYZ} mirror
  symmetry.
\newblock {\em J. Geom. Phys.}, 139:103--138, 2019.
\newblock arXiv:math/1605.00342.

\bibitem{Khare-WintenbergerI}
Chandrashekhar Khare and Jean-Pierre Wintenberger.
\newblock Serre's modularity conjecture. {I}.
\newblock {\em Invent. Math.}, 178(3):485--504, 2009.

\bibitem{Khare-WintenbergerII}
Chandrashekhar Khare and Jean-Pierre Wintenberger.
\newblock Serre's modularity conjecture. {II}.
\newblock {\em Invent. Math.}, 178(3):505--586, 2009.

\bibitem{LMFDB}
The {LMFDB Collaboration}.
\newblock The {L}-functions and modular forms database.
\newblock \url{https://www.lmfdb.org}, 2024.
\newblock [Online; accessed 12 April 2024].

\bibitem{MNOP1}
D.~Maulik, N.~Nekrasov, A.~Okounkov, and R.~Pandharipande.
\newblock Gromov-{W}itten theory and {D}onaldson-{T}homas theory. {I}.
\newblock {\em Compos. Math.}, 142(5):1263--1285, 2006.
\newblock arXiv:math.AG/0312059.

\bibitem{Maulik-Toda-Inventiones-2018}
Davesh Maulik and Yukinobu Toda.
\newblock Gopakumar-{V}afa invariants via vanishing cycles.
\newblock {\em Invent. Math.}, 213(3):1017--1097, 2018.
\newblock arXiv:math/1610.07303.

\bibitem{Miranda-Persson-Extremal}
Rick Miranda and Ulf Persson.
\newblock On extremal rational elliptic surfaces.
\newblock {\em Math. Z.}, 193(4):537--558, 1986.

\bibitem{Morishige-multi-banana}
Nina Morishige.
\newblock {Genus zero Gopakumar-Vafa invariants of Multi-Banana
  configurations}.
\newblock arXiv:math/2102.07965.

\bibitem{Morrison-LookingGlass}
David~R. Morrison.
\newblock Through the looking glass.
\newblock In {\em Mirror symmetry, {III} ({M}ontreal, {PQ}, 1995)}, volume~10
  of {\em AMS/IP Stud. Adv. Math.}, pages 263--277. Amer. Math. Soc.,
  Providence, RI, 1999.
\newblock arXiv:math/9705028.

\bibitem{Pardon-Universal-Curve-Counting-2023}
John Pardon.
\newblock {Universally counting curves in Calabi--Yau threefolds}.
\newblock arXiv:math/2308.02948.

\bibitem{Schoen-88}
Chad Schoen.
\newblock On fiber products of rational elliptic surfaces with section.
\newblock {\em Math. Z.}, 197(2):177--199, 1988.

\bibitem{Yui-with-Verrill-appendix}
Noriko Yui.
\newblock Update on the modularity of {C}alabi-{Y}au varieties.
\newblock In {\em Calabi-{Y}au varieties and mirror symmetry ({T}oronto, {ON},
  2001)}, volume~38 of {\em Fields Inst. Commun.}, pages 307--362. Amer. Math.
  Soc., Providence, RI, 2003.
\newblock With an appendix by Helena Verrill.

\end{thebibliography}
\bibliographystyle{plain}

\end{document}